\documentclass{article}

\usepackage{graphicx} 
\usepackage[utf8]{inputenc}
\usepackage{float}
\usepackage{url}

\usepackage{amssymb,amsthm}

\usepackage[margin =1in]{geometry}
\usepackage{thm-restate}
\usepackage{enumerate}
\usepackage{subcaption}
\usepackage[noend]{algpseudocode}
\usepackage{cite}
\usepackage[algo2e,ruled,vlined,linesnumbered]{algorithm2e}
\usepackage{breqn}
\usepackage{hyperref}
\usepackage{cleveref}
\newtheorem{theorem}{Theorem}

\hypersetup{
    colorlinks=true,
    linkcolor=blue,
    filecolor=magenta,      
    urlcolor=cyan,
    pdftitle={Overleaf Example},
    pdfpagemode=FullScreen,
    }
\usepackage{bbm}

\def\det{\mathrm{det}}
\def\Tr{\mathrm{Tr}}

\def\Reals{\mathbb{R}}

\def\cY{\mathcal{Y}}

\def\dim{\mathrm{dim}}
\def\span{\mathrm{Span}}
\def\Null{\mathrm{Null}}

\def\col{\mathrm{col}}

\def\bb{\mathbf{b}}
\def\bg{\mathbf{g}}
\def\bp{\mathbf{p}}
\def\bq{\mathbf{q}}
\def\bs{\mathbf{s}}
\def\bx{\mathbf{x}}
\def\by{\mathbf{y}}
\def\bz{\mathbf{z}}
\def\bA{\mathbf{A}}
\def\bC{\mathbf{C}}
\def\bI{\mathbf{I}}
\def\bM{\mathbf{M}}
\def\bP{\mathbf{P}}
\def\bQ{\mathbf{Q}}
\def\bS{\mathbf{S}}
\def\bzero{\mathbf{0}}
\newtheorem{lemma}{Lemma}
\newtheorem{assumption}{Assumption}
\newtheorem{corollary}{Corollary} 
\usepackage{amsmath}
\usepackage{orcidlink}

\usepackage[normalem]{ulem} 

\newcommand{\replace}[2]{#2} 
\newcommand{\replacemath}[2]{#2} 

\usepackage{marginnote}
\normalmarginpar

  \makeatletter
\newcounter{banditproblem}
\newenvironment{banditproblem}[1][htb]
  {
   \let\c@algocf\c@banditproblem
   \begin{algorithm2e}[#1]%
  }{\end{algorithm2e}}
\makeatother

\crefname{bp}{Bandit Problem}{Bandit Problems}

\title{Augmenting Subspace Optimization Methods with Linear Bandits}
\author{Matt Menickelly\footnote{Argonne National Laboratory, 9700 S. Cass Ave., Lemont, IL 60439, USA.}\orcidlink{0000-0002-2023-0837}}

\begin{document}

\maketitle

\begin{abstract}
 We consider the framework of methods for unconstrained minimization that are, in each iteration, restricted to a model that is only a valid approximation to the objective function on some affine subspace containing an incumbent point. 
 These methods are of practical interest in computational settings where derivative information is either expensive or impossible to obtain. 
 Recent attention has been paid in the literature to employing randomized matrix sketching for  generating the affine subspaces within this framework. 

We consider a relatively straightforward, deterministic augmentation of such a generic subspace optimization method.
In particular, we consider a sequential optimization framework where actions consist of one-dimensional linear subspaces and rewards consist of (approximations to) the magnitudes of directional derivatives computed in the direction of the action subspace. 
Reward maximization in this context is consistent with maximizing lower bounds on descent guaranteed by first-order Taylor models. 
This sequential optimization problem can be analyzed through the lens of dynamic regret. 
We modify an existing linear upper confidence bound (UCB) bandit method and prove sublinear dynamic regret in the subspace optimization setting.
We demonstrate the efficacy of employing this linear UCB method in a setting where forward-mode algorithmic differentiation can provide directional derivatives in arbitrary directions and in a derivative-free setting. 
For the derivative-free setting, we propose \texttt{SS-POUNDers}, an extension of the derivative-free optimization method \texttt{POUNDers} that employs the linear UCB mechanism to identify promising subspaces. 
Our numerical experiments suggest a preference, in either computational setting, for employing a linear UCB mechanism within a subspace optimization method. 
\end{abstract}

\section{Introduction}
We consider the minimization of 
 an unconstrained objective function $f:\Reals^d\to\Reals$; that is, we mean to solve
 \begin{equation}\label{eq:minf} 
 \displaystyle\min_{\bx\in\Reals^d} f(\bx).
 \end{equation}
\replace{We make the following assumption on $f$, which will hold throughout this manuscript.}{Throughout the manuscript, we make the following assumption on $f$.}

\begin{assumption}\label{ass:f}
The objective function $f$ satisfies the following. 
\begin{itemize}
\item $f:\Reals^d\to\Reals$ is continuously differentiable on a level set $$L(\bx_0):= \{\bz\in\Reals^d: f(\bz)\leq f(\bx_0)\}$$. 
\item $f$ is bounded below on $L(\bx_0)$; that is,  there exists $\bx_*\in\Reals^d$ such that $f(\bx_*)\leq f(\bx)$ for all $\bx\in L(\bx_0)$.
\item $\nabla f$ is globally Lipschitz with constant $L$ on the level set $L(\bx_0)$; that is, 
$$\|\nabla f(\bx)-\nabla f(\by)\|\leq L\|\bx-\by\|$$ for all $\bx,\by \in L(\bx_0)$. 
\end{itemize}
\end{assumption}


We note that \Cref{ass:f} makes no particular assumptions about convexity, which informs our choices of algorithms throughout this manuscript.   
Importantly, \Cref{ass:f} suggests that because gradients are well defined everywhere, first-order optimization methods are appropriate for the identification of first-order stationary points of $f$ in \eqref{eq:minf}. 

In some areas of computational science, the computation of derivatives of objective function $f$ may be prohibitive even when \Cref{ass:f} is satisfied. This is traditionally the scope of derivative-free optimization (DFO) methods; see, for example, \cite{AudetHare2017, Conn2009a, LMW2019AN}. 
Alternatively, for such functions $f$ that are provided as computer codes, algorithmic differentiation (AD) can provide numerically accurate derivative computations. 
A celebrated result (see, e.g., \replace{\cite{Griewank_2003}[Section 3.2]}{\cite[Section 3.2]{Griewank_2003}}) proves that, in the absence of constraints on memory, reverse-mode AD can compute a numerical gradient of a function in a small factor times the wall-clock time of performing a zeroth-order function evaluation. 
Of course, these memory assumptions are unrealistic; and in practice, the computation of derivatives via reverse-mode AD is challenging, if not prohibitive.
For a more thorough discussion of such challenges, see, for example,~\replace{\cite{huckelheim20}[Section 1]}{\cite[Section 1]{huckelheim20}}.
In situations where reverse-mode AD is too difficult, however, forward-mode AD is sometimes tractable. 
Informally, forward-mode AD computes directional derivatives (in potentially arbitrary, that is, not coordinate-aligned, directions\replace{}{; see, for instance the \texttt{jvp} routine of \texttt{JAX}\cite{jax2018github}, which permits arbitrary tangent vectors.}).
The cost of a single directional derivative in this mode is approximately the cost of a single zeroth-order evaluation of $f$.
This is, of course, an informal statement; see \cite{huckelheim20} for practical insight into more precise quantification of the true costs of forward-mode AD. 
Thus, one may view forward-mode AD as an \emph{oracle} for computing directional derivatives (in arbitrary directions) at a  cost similar to that of performing finite difference approximations of derivatives but with generally superior accuracy. 

In computational settings that preclude analytic derivatives and efficient reverse-mode AD, \emph{\replace{}{model-based} subspace optimization methods} are of particular interest for the solution of \eqref{eq:minf}, \replace{}{especially when the dimension $d$ is large,} and have seen recent theoretical growth. 
For the sake of exposition, we broadly define \replace{}{model-based} subspace methods as any iterative method of minimization that can be described by the informal statement provided in Framework~1. 

\begin{banditproblem}[H]
    \caption{Informal definition of a \replace{}{model-based} subspace optimization method \label{bp:krounds}}
         \SetAlgoNlRelativeSize{-4}
         \textbf{Input: } Initial point $\bx_0\in\Reals^d$, number of iterations $K$, function $f$ satisfying \Cref{ass:f}. \\
         \For{$k=0,1,\dots, K-1$}{
         Choose a subspace dimension $p_k\leq d$, and choose a matrix $\bS_k\in\Reals^{d\times p_k}$. \\
         Construct some model, $m_k:\Reals^{p_k}\to\Reals$, such that $m_k(\bz)$ approximates $f(\bx_k + \bS_k \bz)$ for all $\bz\in\Reals^{p_k}$ in a neighborhood of $\bzero\in\Reals^{p_k}$.  \\
         Approximately minimize $m_k$ in a neighborhood of $\bzero\in\Reals^{p_k}$ to obtain $\bz_*\in \Reals^{p_k}$. \\
         Update the incumbent $\bx_{k+1}\gets \bx_k + \bS_k \bz_*$, provided $f(\bx_k + \bS_k \bz_*)$ yields sufficient decrease. \\
         }
\end{banditproblem}

Deterministic \replace{}{model-based} subspace methods have been studied for decades; Krylov methods, and more specifically nonlinear conjugate gradient methods, fit within Framework~1 by implicitly choosing $\bS_k$ to define a (typically) two-dimensional subspace based on past gradient evaluations. 
(L)BFGS methods implicitly generate $\bS_k$ such that the span of the columns of $\bS_k$ is equal to the span of a (limited) memory of gradients and past incumbent/gradient displacements. 
Block coordinate descent methods can be viewed as choosing $\bS_k$ 
as a (potentially deterministic) subset of columns of a $d$-dimensional identity matrix. 
The former two types of subspace methods require the computation of (full-space) gradients, so these methods are not within the scope of our computational setting. 
Block coordinate descent methods do fit in our computational setting but are not the topic of this paper.  

Recent developments have seen the application of \emph{randomized} matrix sketching~\cite{mahoney2011randomized, woodruff2014sketching, martinsson2020randomized, murray2023randomized} as a means to choose $\bS_k$ within algorithms in Framework~1\replace{.}{, effectively choosing $\bS_k$ as a \emph{sketch matrix}, a random matrix governed by a distribution such that statistical guarantees can be made about how $\bS_k$ acts on a vector (here, $\nabla f(\bx_k)$).} 
In a computational setting that permits forward-mode AD, the local model $m_k$ of Framework~1 could be simply a first-order Taylor model in the affine subspace defined by $\bx_k$ and $\bS_k$. Computing $\bS_k^\top\nabla f(\bx_k)$ amounts to computing a (relatively small) number of directional derivatives in the randomized directions given in the columns of $\bS_k$.
In the derivative-free setting, an arbitrary model $m_k$ is constructed in the (relatively small) affine subspace defined by $\bx_k$ and $\bS_k$; $m_k$ might exhibit the property that the zeroth- and first-order errors made by the model are comparable to those made by a first-order Taylor model. 
In either the forward-mode AD setting or the derivative-free setting, and as we will expound upon, the theoretical success of randomized subspace methods hinges upon some notion that the randomly selected subspace in each iteration exhibits, with sufficiently high probability, a sufficiently large normalized inner product with the gradient at the incumbent~\replace{; that is, roughly speaking,}{see, e.g., \cite[Assumption 4]{cartis2022randomised},~\cite[Assumptions 5,6]{cartis2023scalable},~\cite[Assumption 5, Lemma 4]{chen2024q}
,~\cite[Assumptions 2,3]{dzahini2024stochastic}.
All of these works assume, roughly speaking, that with high probability and for some $q\in(0,1)$, that}
\begin{equation}\label{eq:rough_desire}
\replacemath{\|(\bS_k^\top \bS_k)^{-1} \bS_k^\top \nabla f(\bx_k)\| \approx \|\nabla f(\bx_k)\|.}{\displaystyle\frac{\|(\bS_k^\top \bS_k)^{-1} \bS_k^\top \nabla f(\bx_k)\|}{\|\nabla f(\bx_k)\|} > q;}\end{equation}
\replace{}{Here, we have made the expositional choice to assume that $\bS_k$ has full column rank so that a matrix with $p_k$ many columns always defines a $p_k$-dimensional subspace. Hence, $\bS_k^\top\bS_k$ is always invertible, and so the normalization by the matrix inverse in \eqref{eq:rough_desire} is well-defined. 
One may notice that the normalization does not explicitly appear in the previously cited works, but is instead handled by assuming a (probabilistic) bound on $\|\bS_k\|$.}

Providing a guarantee like \eqref{eq:rough_desire} typically relies on some variant of the Johnson--Lindenstrauss lemma \cite{johnson1984extensions}, a surprising result that demonstrates, roughly speaking, that in high dimensions, a low-dimensional subspace defined by $\bS_k$ (with $p_k$ independent of $d$) drawn from an appropriate distribution achieves a property resembling \eqref{eq:rough_desire} with high probability. 

\subsection{Contributions}
In this manuscript we propose a novel \emph{deterministic} method for choosing subspaces $\bS_k$ within subspace optimization methods like those described by Framework~1.  
Instead of directly employing randomized sketches in an effort to (probabilistically) \replace{align the random subspace with a gradient}{achieve a condition like \eqref{eq:rough_desire}}, we demonstrate how \replace{this alignment problem can be described through the lens of}{achieving \eqref{eq:rough_desire} can be performed by} \emph{linear bandit methods} in a nonstationary environment. 
Because \Cref{ass:f} guarantees that the gradient of $f$ is Lipschitz continuous, we treat the incumbent update in a generic subspace optimization method as an exogenous process resulting in bounded drifts in a  gradient. 

More specifically, we adapt a linear upper confidence bound (UCB) bandit method~\cite{auer2002using, abbasi2011improved} to choose each unit-length one-dimensional sketch $\bS_k\replacemath{}{\in\mathbb{R}^d}$ so that $\bS_k$ is maximally aligned with $\nabla f(\bx_k)$; that is, we aim to maximize $|(\bS_k^\top \bS_k)^{-1} \bS_k^\top \nabla f(\bx_k)|$ in every iteration. 
The linear UCB method maintains a confidence ellipsoid describing a belief on the dynamically changing value of $\nabla f(\bx_k)$ by fitting a linear regression model using a history of observed (approximations to) $\bS_k^\top \nabla f(\bx_k)$. 
The maximizer over this confidence ellipsoid of an inner product with a unit-norm vector is chosen on each iteration as the one-dimensional sketch $\bS_k$. 
Such a UCB method ensures a balance between \emph{exploration} of subspaces---the focus of randomized sketching approaches to subspace optimization methods in the literature---and \emph{exploitation} of subspaces that have recently been observed to exhibit large gradients. 
The linear UCB method is introduced and analyzed in \Cref{sec:method}.

In \Cref{sec:method} we also demonstrate sublinear dynamic regret \cite{besbes2014stochastic, besbes2015non} of this linear UCB method in terms of its ability to approximate the dynamically changing gradient $\nabla f(\bx_k)$.  
Because of our choice to view the optimization dynamics as an exogenous process affecting the gradient~$\nabla f(\bx_k)$, 
we stress that this is not a regret bound for the optimization algorithm. 
However, our regret analysis will elucidate how the selection of a particular subspace optimization algorithm will affect the regret incurred by the linear UCB method. 
As an illustration, we will consider perhaps the simplest possible optimization algorithm that could fit into Framework~1---a subspace variant of \replace{fixed stepsize gradient descent that employs randomized sketching}{gradient descent with a backtracking line search}---to make these connections clear. 
In particular, given a sequence \replace{of matrices $\{\bS_k\}$}{$\{\bS_k\}$ of full column rank matrices}, define the associated sequence of embedding matrices~$\{\bP_k := \bS_k (\bS_k^\top \bS_k)^{-1}\bS_k^\top\} \subset \Reals^{d\times d}$\replace{, and let $\{\alpha_k\}$ denote a strictly positive sequence of stepsizes.
With these choices made, we define a \emph{subspace gradient descent method} via the iteration}{. Then, we consider the following algorithm: }


\begin{algorithm2e}[H]
 \SetAlgoNlRelativeSize{-4}

 \textbf{(Inputs)} 
 Initial point $\bx_1$,
 horizon $K$.
   
\textbf{(Initializations)} 
Backtracking parameter $\beta\in(\frac{1}{2},1)$.

\For{$k=1,2,\dots, K$}{
\textbf{(Obtain (random) sketch)} 

Receive $\bS_k \in\Reals^{d\times p_k}$ and call oracle to compute $\bS_k^\top \nabla f(\bx_k)$. \label{line:call_oracle1}

\textbf{(Determine next incumbent by linesearch)}\label{line:stepsize_adjust}

$\alpha \gets 1.$

$\bx_{+}\gets \bx_k - \alpha \bP_k \nabla f(\bx_k)$.

\While{$f(\bx_{+}) \geq f(\bx_k) - \alpha\|\bP_k\nabla f(\bx_k)\|^2$ \label{line:ls_while}} 
{

$\alpha\gets \beta\alpha$

$\bx_{+}\gets \bx_k - \alpha \bP_k \nabla f(\bx_k)$.
}

$\bx_{k+1} \gets \bx_{+}.$
} 
\caption{\label{alg:linesearch} Subspace Gradient Descent with Backtracking Line Search. }
\end{algorithm2e}
\color{black} 
\replace{}{\Cref{alg:linesearch} is among the simplest algorithms that could fit into Framework~1 because the subspace model $m_k$ employed in each iteration is a first-order Taylor model restricted to the subspace defined by $\bS_k$, and a quantifiable notion of sufficient decrease from $m_k$ is classically guaranteed by the backtracking line search for $f$ satisfying \Cref{ass:f}.}

We are particularly concerned in this manuscript with practicalities; and so, motivated by various pieces of the analysis in \Cref{sec:method}, we propose practical extensions in \Cref{sec:considerations}. 
We conclude this paper by proposing and testing our linear UCB mechanism in \replace{an algorithm that can be viewed as a specific instance of a quadratic regularization method with randomized sketching proposed in \cite{cartis2022randomised}}{the practical variant of \Cref{alg:linesearch}} in a setting where we simulate the computation of forward-mode AD gradients\replace{; the method is essentially that described by (3)  but with a dynamically adjusted stepsize $\alpha_k$.}{.}
In a derivative-free setting, we propose a simple extension of the DFO method \texttt{POUNDers} \cite{SWCHAP14} that employs our linear UCB mechanism to select subspaces.

\subsection{Related Work}
The theoretical and practical development of methods that have the general form of Framework~1 has enjoyed much recent attention.
Matrix sketching \cite{mahoney2011randomized, woodruff2014sketching, martinsson2020randomized, murray2023randomized}, which entails the random generation of $\bS_k$ in Framework~1 from some distribution on matrices, has emerged as a theoretical tool for analyzing such methods, often supported by the powerful Johnson--Lindenstrauss lemma \cite{johnson1984extensions}. 
In \cite{kozak2019stochastic} the authors propose a method in Framework~1 that looks much like \replace{the iteration (3)}{the subspace gradient descent method in \Cref{alg:linesearch}, but with a fixed step size}.
The authors prove convergence of the method under standard assumptions about convexity of $f$;
in a follow-up paper~\cite{kozak2021stochastic}, the same authors considered the same iteration but with $\nabla f(\bx_k)$ approximated by finite differences. 
The works \cite{gower2019rsn, yuan2022sketched, fuji2022randomized} and
\cite{hanzely2020stochastic, zhao2024cubic} propose and analyze randomized Newton methods and stochastic subspace cubic Newton methods, respectively, that extend \replace{(3)}{a subspace gradient descent iteration} by also using second-derivative information projected onto subspaces defined by $S_k$. 
Pilanci and Wainwright~\cite{pilanci2017newton} also similarly consider randomized Newton methods in the case where the objective Hessian admits a matrix square root everywhere on its domain. 
Cartis et al.~\cite{cartis2022randomised} consider imposing various globalization strategies, such as trust-region and quadratic regularization, into algorithms that fall into Framework~1; the authors also specialize their methods to nonlinear least-squares objectives. Miyashi et al.~\cite{miyaishi2024subspace} consider a quasi-Newton approximation of projected Hessian information to be used with projected gradients~$\bS_k^\top\nabla f(\bx_k)$, and Nozawa et al.~\cite{nozawa2023randomized} consider variants of \replace{(3)}{the subspace gradient descent iteration} for problems with nonlinear inequality constraints. 
Grishchenko et al.~\cite{grishchenko2021proximal} effectively consider the \replace{iteration (3)}{subspace gradient descent iteration} within a proximal point framework to handle an additional convex nonsmooth term in \eqref{eq:minf}. 

Although not precisely based on sketching, many randomized algorithms have been proposed based on a related idea thoroughly expounded in \cite{Nesterov2015} that was introduced as early as in \replace{\cite{polyakbook}[Chapter 3.4]}{\cite[Chapter 3.4]{polyakbook}}. 
In such algorithms, one effectively draws one-dimensional sketches $\bS_k\in\Reals^d$ from a zero-mean, typically Gaussian, distribution with arbitrary covariance and then computes an unbiased estimator \replace{of~$\nabla f(\bx_k)$,~$(\nabla\bS_k^\top \nabla f(\bx_k))\bS_k$}{$\bS_k^\top \nabla f(\bx_k))\bS_k$ of $\nabla f(\bx_k)$}. 
Randomized methods depending on a forward-mode AD oracle for computing $\bS_k^\top \nabla f(\bx_k)$ in this unbiased estimator have also attracted recent attention in the AD community \cite{baydin2022gradients, shukla2023randomized}, with a particular interest in application to fine-tuning of large language models \cite{malladi2023fine, zhang2024revisiting}.

\replace{S}{Model-based s}ubspace methods based on sketching have also inspired research in DFO. 
If $f$ is not readily amenable to AD, then one can approximate a given $\bS_k^\top \nabla f(\bx_k)$ by finite differencing at the cost of $\mathcal{O}(p_k)$ many function evaluations. 
This idea of approximating local projected gradient information lends itself naturally to model-based DFO methods, and this idea has been explored thoroughly in \cite{cartis2023scalable, cartis2024randomized} and the previously mentioned \cite{kozak2021stochastic}.  
Related analyses of subspace model-based DFO were performed in \cite{chen2024q, hare2025expected}.
Subspace model-based DFO was extended to stochastic objectives in \cite{dzahini2024stochastic}\replace{.}{and to more general noisy objectives in \cite{dzahini2025noise}.} 
\replace{}{Although not model-based methods, randomized subspace techniques have also been explored in the context of direct-search methods, another well-studied class of methods in DFO \cite{roberts2023direct, dzahini2024direct}.}
We additionally mention that Nesterov and Spokoiny~\cite{Nesterov2015} also introduced a finite-difference-based (and hence biased) estimator for the derivative-free setting, which has  inspired much  research; see \cite{berahas2022theoretical} and references therein.

While the works cited so far concern randomized methods (i.e., $\bS_k$ is drawn from some distribution), we remind the reader that our broad definition in Framework~1 also encompasses deterministic methods. 
Block-coordinate descent methods are widely found in the literature  (see in particular \cite{nesterov2012efficiency, wright2015coordinate, richtarik2014iteration}) and need not be randomized. 
A vast literature on conjugate gradient methods also exists; but we draw special attention to \cite{yuan1995subspace}, which proposed a method that iteratively minimized a model of the objective restricted to a two-dimensional search space defined by the span of $\nabla f(\bx_k)$ and the most recent displacement vector between incumbents. 
In our notation, this is tantamount to including $\nabla f(\bx_k)$ as a column of $\bS_k$. 
Obviously, proposing to use such a sequence of $\{\bS_k\}$ in our setting, where we assume the full computation of $\nabla f(\bx_k)$ is expensive, is nonsensical but would  satisfy \eqref{eq:rough_desire}, which is the overarching goal of the UCB method we propose using in this manuscript. 
The identification of many well-known minimization methods effectively amounting to iterative minimization of local models along gradient-aligned subspaces is the subject of \cite{yuan2007subspace, yuan2014review}; this idea has permeated much development across the field. 
We also note that the Ph.D. thesis \cite{zhangthesis} considered, essentially, a derivative-free variant of the method of \cite{yuan1995subspace} that would periodically recompute a new approximation to $\nabla f(\bx_k)$ for use in determining a subspace. 

We comment that the linear UCB mechanism we propose in this manuscript is meant to \emph{locally} predict subspaces exhibiting the most variation, which are, in many applications, likely to change as the incumbent $\bx_k$ changes\replace{.}{; as will be seen in \Cref{alg:ss_ucb}, an employed memory parameter $M$ will guarantee that the UCB method only retains the $M$ most recent sketched gradients in memory, justifying this locality.} 
We must mention that for problems that exhibit \emph{global} subspaces of variation, there has been much research in identifying such so-called active subspaces \cite{Constantine2015}. 
In fact, some recent work has made explicit connections between sketching algorithms and identifying such global active subspaces \cite{cartis2023global, cartis2024learning}.
While our linear UCB mechanism is intuitively also likely to capture such low effective dimensionality, provided it exists for a given instance of $f$ in \eqref{eq:minf}, our linear UCB mechanism is intended to identify, dynamically, \emph{local} subspaces of variation. 
Nonetheless, in our numerical experiments we will test the hypothesis that the linear UCB mechanism should identify active subspaces. 

\section{Preliminaries}
To motivate our development, 
we begin with a lemma \replace{concerning a subspace gradient descent method defined via the iteration (3)}{guaranteeing the sufficient decrease attained in each iteration of \Cref{alg:linesearch}.} 

\begin{lemma}\label{lem:bounds_on_decrease}
Let \Cref{ass:f} hold, in particular let $L$ be a global Lipschitz constant for $\nabla f$. 
    \replace{}{Let a sequence of full column rank matrices $\{\bS_k\}$ be given.}
    For all $k=0,1,\dots, K$, the sequence of $\{\bx_k\}$ generated by \replace{the iteration (3)}{\Cref{alg:linesearch}} satisfies
    $$f(\bx_k) - f(\bx_{k+1})  \geq \replacemath{\frac{1}{2L}}{\frac{2\beta - 1}{2L}}\|(\bS_k^\top \bS_k)^{-\frac{1}{2}} \bS_k^\top \nabla f(\bx_k)\|^2, $$
  where $\beta$ is the backtracking parameter from \Cref{alg:linesearch}.
\end{lemma}

\begin{proof}
    \replace{}{Suppose the while loop beginning in \Cref{line:ls_while} of \Cref{alg:linesearch}  runs through~$N:=~\lceil \log_{\beta}(\min\{\beta, 1/L\})\rceil$ many iterations so that $\alpha_k \in (\beta/L, 1/L]$.}
    By standard Taylor error arguments and using \Cref{ass:f}, we have
    $$
    \begin{array}{rl} f(\bx_{k+1}) & \leq 
    f(\bx_k) + \langle \nabla f(\bx_k), \bx_{k+1}- \bx_k\rangle + \frac{L}{2}\|\bx_{k+1}-\bx_k\|^2\\
    & = 
    f(\bx_k) + \langle \nabla f(\bx_k), -\replacemath{\frac{1}{L}}{\alpha_k} \bP_k\nabla f(\bx_k)\rangle + \frac{L}{2}\|-\replacemath{\frac{1}{L}}{\alpha_k}P_k\nabla f(\bx_k)\|^2\\
    & = 
    f(\bx_k) - \replacemath{\frac{1}{L}}{\alpha_k} \nabla f(\bx_k)^\top \bS_k (\bS_k^\top \bS_k)^{-1} \bS_k^\top \nabla f(\bx_k) \\
    & + \replacemath{\frac{1}{2L}}{\frac{L\alpha_k^2}{2}}\nabla f(\bx_k)^\top \bS_k (\bS_k^\top \bS_k)^{-1}\bS_k^\top \bS_k (\bS_k^\top \bS_k)^{-1}\bS_k^\top \nabla f(\bx_k)\\
    & 
    \leq
    f(\bx_k) - \frac{\beta}{L} 
    \nabla f(\bx_k)^\top \bS_k (\bS_k^\top \bS_k)^{-1} \bS_k^\top \nabla f(\bx_k) \\
    & 
    + \frac{1}{2L}\nabla f(\bx_k)^\top \bS_k (\bS_k^\top \bS_k)^{-1}\bS_k^\top \bS_k (\bS_k^\top \bS_k)^{-1}\bS_k^\top \nabla f(\bx_k)
    \\
    \color{black}
    & = 
    f(\bx_k) - \replacemath{\frac{1}{2L}}{\frac{2\beta - 1}{2L}}  \nabla f(\bx_k)^\top \bS_k (\bS_k^\top \bS_k)^{-1} \bS_k^\top \nabla f(\bx_k)\\
    & = 
    f(\bx_k) - \replacemath{\frac{1}{2L}}{\frac{2\beta - 1}{2L}}\|(\bS_k^\top \bS_k)^{-\frac{1}{2}} \bS_k^\top \nabla f(\bx_k)\|^2\replacemath{,}{}
    \end{array}
    $$
    \replace{as intended.}{
    Thus, if the while loop of \Cref{alg:linesearch} terminates before $N$ iterations, then~$\alpha_k > \beta/L > (2\beta - 1)/2L$, and so 
    $$f(\bx_{k+1}) \leq f(\bx_k) - \alpha_k\|\bP_k\nabla f(\bx_k)\|^2 \leq f(\bx_k)-\frac{2\beta - 1}{2L}\|\bP_k\nabla f(\bx_k)\|^2,$$
    proving the claim. 
    } 
\end{proof}

The next lemma is not about any particular optimization method but is an observation about a general $\bS_k$. 
\replace{}{In particular, \Cref{lem:preserved} demonstrates the intuitive result that the norm of any projection of $\nabla f(\bx_k)$ onto $\bS_k$ is necessarily upper bounded by the norm of $\nabla f(\bx_k)$, and equality can only be attained if $\nabla f(\bx_k)$ lies in the span of~$\bS_k$.}

\begin{lemma}\label{lem:preserved}
    For any $\bS_k\in\Reals^{d\times p_k}$ with full column rank, 
    $$\|(\bS_k^\top \bS_k)^{-\frac{1}{2}} \bS_k^\top \nabla f(\bx_k)\|^2 \leq \|\nabla f(\bx_k)\|^2,$$
    with equality attained if $\nabla f(\bx_k)\in\span(\bS_k)$. 
    Here, we say $\bz\in\span(\bS_k)$ provided there exists $\by\in\Reals^{p_k}$ such that $\bS_k \by = \bz.$
\end{lemma}

\begin{proof}
    Because $\bP_k=\bS_k(\bS_k^\top \bS_k)^{-1}\bS_k^\top$ is a projection matrix, its eigenvalues are all either 0 or 1. 
    Therefore, 
    $\bI_d - \bS_k(\bS_k^\top \bS_k)^{-1}\bS_k^\top \succeq 0$,
    and so 
    $$
\|\nabla f(\bx_k)\|^2 - \|(\bS_k^\top \bS_k)^{-\frac{1}{2}} \bS_k^\top \nabla f(\bx_k)\|^2 = 
\nabla f(\bx_k)^\top (\bI_d - \bS_k(\bS_k^\top \bS_k)^{-1}\bS_k^\top)\nabla f(\bx_k) \geq 0, $$
proving the inequality. 

Now, suppose $\nabla f(\bx_k)\in \span(\bS_k)$. 
Then, without loss of generality, we may assume that the columns of $\bS_k$ are $\nabla f(\bx_k)/\|\nabla f(\bx_k)\|$ and a set of $p_k-1$ unit vectors in $\Reals^d$, mutually orthogonal with $\nabla f(\bx_k)$. 
Then,
$$\|(\bS_k^\top \bS_k)^{-\frac{1}{2}} \bS_k^\top \nabla f(\bx_k)\|^2
= 
\|\bS_k^\top \nabla f(\bx_k)\|^2
= 
\left| \displaystyle\frac{\nabla f(\bx_k)}{\|\nabla f(\bx_k)\|}^\top\nabla f(\bx_k)\right|^2
= \|\nabla f(\bx_k)\|^2. 
$$
\end{proof}

Combining \Cref{lem:bounds_on_decrease} and \Cref{lem:preserved}, 
we see that the lower bound on \replace{}{per-iteration} function decrease guaranteed by \replace{(3)}{\Cref{alg:linesearch}} proven in \Cref{lem:bounds_on_decrease} is maximized when $\bS_k$ is \emph{well-aligned} with $\nabla f(\bx_k)$ in the sense that \eqref{eq:rough_desire} is satisfied, that is,~$\|(\bS_k^\top \bS_k)^{-\frac{1}{2}}\bS_k^\top \nabla f(\bx_k)\|^2 \approx \|\nabla f(\bx_k)\|^2$. 
It is precisely this sort of alignment that appears in multiple analyses of methods in Framework~1. 
In fact, \cite{cartis2022randomised, cartis2023scalable, dzahini2024stochastic} all impose versions of the following two assumptions, which capture precisely this intuition of seeking to minimize the lower bound in \Cref{lem:bounds_on_decrease}. 

\begin{assumption}\label{ass:sketch}
Let $\replacemath{\delta > 0}{\delta \in [0,1)}$ and $\replacemath{q>0}{q\in(0,1)}$ be given. 
The sequence $\{\bS_k\}$ is generated in such a way that in the $k$th iteration, the inequality
\begin{equation}\label{eq:well-aligned}
\|\bS_k^\top \nabla f(\bx_k)\|^2 \geq q\|\nabla f(\bx_k)\|^2
\end{equation}
holds with probability at least $1-\delta$. 
When the event \eqref{eq:well-aligned} is realized, we say that \emph{$\bS_k$ is $q$-well-aligned with $\nabla f(\bx_k)$}. 
\end{assumption}

\begin{assumption}\label{ass:sketch2}
There exists $S_{\max}>0$ such that the largest eigenvalue of~$\bS_k^\top \bS_k$,~$\lambda_{\max}(\bS_k^\top \bS_k)$, satisfies $\lambda_{\max}(\bS_k^\top \bS_k) > S_{\max}$ for all $k$. 
\end{assumption}

\Cref{ass:sketch} has a clear relationship with subspace embedding in randomized linear algebra \cite{mahoney2011randomized,woodruff2014sketching, martinsson2020randomized, murray2023randomized}. 
We recall here that there exist multiple sketching procedures capable of satisfying \Cref{ass:sketch}, and we outline three such means here. 
\begin{enumerate}
\item $\bS_k$ is a \emph{Gaussian matrix}; in other words,  each entry of $\bS_k$ is independently drawn from $\mathcal{N}(0,1/p_k)$. 
In this case, given both $\delta$ and $q$, \Cref{ass:sketch} will hold provided $p_k\in\Omega\left(\frac{|\log(\delta)|}{(1-q)^2}\right)$ \footnote{Throughout this manuscript, we will define $\Omega(\cdot)$, $\Theta(\cdot)$, and $\mathcal{O}(\cdot)$, as in \cite{knuth1976big}.} 
One can use matrix Bernstein inequality to derive high probability bounds to put $\lambda_{\max}(\bS_k^\top \bS_k)\in\mathcal{O}(d/p_k)$. 
Thus, while Gaussian sampling cannot guarantee that \Cref{ass:sketch2} holds, we can claim via an union bound over all $K$ matrices that \Cref{ass:sketch2} will hold with high probability. 
For more details on these fairly well-known results, see, for example, \replace{\cite{martinsson2020randomized}[Section 8.6]}{\cite[Section 8.6]{martinsson2020randomized}}. 

\item $\bS_k$ is an \emph{$h_k$-hashing matrix}; in other words, \replace{$\bS_k$ has exactly $h_k$ nonzero entries in each column, with the indices sampled independently and}{each column of $\bS_k$ is instantiated as a vector of zeros, and then exactly $h_k$ entries are chosen at random without replacement. Then, }, each nonzero entry \replace{has}{is assigned} value $\pm 1/\sqrt{h_k}$, each value being selected independently with probability $1/2$. 
In this case, given $\delta$ and $q$, then provided $h_k\in\Theta\left(\frac{|\log(\delta)|}{(1-q)}\right)$ and $p_k\in\Omega\left(\frac{|\log(\delta)|}{(1-q)^2}\right)$, \Cref{ass:sketch} is satisfied \cite{kane2014sparser}.
Coarsely, we can bound 
$\|\bS_k\|\leq\|\bS_k\|_F=\sqrt{d}$, and so 
\Cref{ass:sketch2} is also satisfied with $S_{\max} = d$.

\item $\bS_k$ is the \emph{first $p_k$ columns of a Haar-distributed orthogonal matrix}, scaled by a factor $\sqrt{d/p_k}$.
In this case, given $p_k$ and $q$, 
$\bS_k$ will satisfy \Cref{ass:sketch} provided~$\delta$ is \replace{the value of the cumulative distribution function of a beta distribution with shape parameters $\alpha=p_k/2$ and $\beta=(d-p_k)/2$ evaluated at $q^2p_k/d$}{chosen as $B_{p_k/2, (d-p_k)/2}(q^2p_k/d)$, where $B_{\alpha,\beta}$ denotes the cumulative distribution function of the beta distribution with shape parameters $\alpha$ and $\beta$}; see \cite{kozak2021stochastic}. 
We also immediately obtain that \Cref{ass:sketch2} is satisfied with $S_{\max} = d/p_k$. 
\end{enumerate}
Remarkably, all three of these sketching techniques superficially allow for the subspace dimension $p_k$ to be independent of $d$. This dimension independence is less obvious with the Haar distribution \replace{}{(dimension independence effectively means~$B_{p_k/2, (d-p_k)/2}(q^2p_k/d)\in\Theta(1)$ as a function of $d$)}, but numerical evidence suggests this is the case \cite{kozak2021stochastic}. \replace{}{Dimension independence may instead be directly provable from, e.g.,~\cite[Lemma 6.1]{meckes2019random}.} 

We will later impose \Cref{ass:sketch} and \Cref{ass:sketch2} as minimal assumptions on a random sequence of sketches $\{\bS_k\}$ when considering the specific use of \replace{the subspace gradient descent method (3)}{\Cref{alg:linesearch}} as an optimization algorithm within Framework~1. 
With the goal in mind of satisfying something resembling \eqref{eq:rough_desire} on every iteration of an arbitrary algorithm in Framework~1, we now state our bandit model in Framework~2. 

\begin{banditproblem}[H]
    \caption{Description of linear bandit model\label{bp:krounds}}
         \SetAlgoNlRelativeSize{-4}
         \textbf{Input: } Initial point $\bx_0$, number of rounds $K$, function $f$ satisfying \Cref{ass:f}. \\
         \For{$k=0,1,\dots, K-1$}{
         \textbf{Endogenous Action: } Decision-maker chooses $\bs_k\in\Reals^d$. \\
         \textbf{Exogenous Action: } Environment chooses $p_k\in[0,d]$ and $\bS_k\in\Reals^{d\times p_k}.$ \\
         \textbf{Feedback: } Decision-maker receives (an approximation of) $[\bS_k,  \bs_k]^\top \nabla f(\bx_k)$. \\
         \textbf{State change: } $\bx_k$ is updated according to an \replace{}{iteration of an} algorithm in Framework~1, e.g., \replace{(3)}{an iteration of \Cref{alg:linesearch}} \\
         }
\end{banditproblem}

Framework~2 defines a game consisting of $K$ rounds. 
In the $k$th round of the game, 
the decision-maker must select one linear measurement $\bs_k\in\Reals^d$; we refer to this one-dimensional measurement $\bs_k$ as the decision-maker's \emph{action} from the action space $\Reals^d$. 
The decision-maker then receives directional derivatives in the directions defined by~$\bS_k$ and the selected $\bs_k$; we assume the $p_k$ columns of $\bS_k$
are exogenously determined, for example, by a randomized sketching method. 
The environment then drifts slightly, in that the $k$th iteration of an optimization algorithm, for example, \replace{(3)}{\Cref{alg:linesearch}}, is performed. 

Because $\bx_k$, the state, changes as a result of performing each iteration of an optimization algorithm from Framework~1, the environment is nonstationary. 
Moreover, because we intend to solve an optimization problem \eqref{eq:minf} only once, it is inappropriate to view Framework~2 as a setting amenable to reinforcement learning; however, we do remark that if one is solving many similar optimization problems \eqref{eq:minf} via subspace optimization methods, learning a policy for playing Framework~2 could be interesting, but it is not the subject of this manuscript. 
As a result, it is most appropriate to analyze any method for playing the game in Framework~2 from the perspective of \emph{dynamic regret} \cite{besbes2014stochastic, besbes2015non}.
We are particularly interested in  dynamic regret in \emph{slowly varying} or \emph{drifting} environments \cite{cheung2019learning, cheung2022hedging}. 
We say the environment is slowly varying because the reward function changes continuously with $\bx_k$ due to \Cref{ass:f}, and the state change is defined by an optimization iteration that seeks to asymptotically steer $\nabla f(\bx_k)\to \bzero$. 

We now describe dynamic regret.
We have already specified the actions in the linear bandit problem in Framework~2.
Based on our discussion around \Cref{lem:bounds_on_decrease}, a natural choice of \emph{instantaneous regret} $R_k$ in the $k$th round of Framework~2 is 
\begin{equation}\label{eq:instantaneous_regret_pre}
R_k = \displaystyle\max_{\bS\in\Reals^{d\times p_k}} 
\|(\bS^\top \bS)^{-\frac{1}{2}}\bS^\top\nabla f(\bx_k)\| - \|(\bS_k^\top \bS_k)^{-\frac{1}{2}}\bS_k^\top\nabla f(\bx_k)\|. 
\end{equation}
As is standard in the definition of regret, the quantity $R_k$ in \eqref{eq:instantaneous_regret_pre} is the gap between the best reward the decision-maker could have received  and the reward that is realized. 
By \Cref{lem:preserved}, if the decision-maker chooses the action $\bs_k = \nabla f(\bx_k)$, then the first term is maximized; and so we can simplify \eqref{eq:instantaneous_regret_pre} to 
$$R_k = \|\nabla f(\bx_k)\| - \|(\bS_k^\top \bS_k)^{-\frac{1}{2}}\bS_k^\top\nabla f(\bx_k)\| = 
\|\nabla f(\bx_k)\|_{\bI - \bS_k(\bS_k^\top \bS_k)^{-1}\bS_k^\top}. $$
\replace{}{Throughout this manuscript, we use the notation $\|\cdot\|_{\bM}$ to denote the Mahalanobis norm with respect to an arbitrary positive definite matrix $\bM$, that is, $\|\bx\|_{\bM} = (\bx^\top \bM\bx)^{1/2}$.}
Thus, the decision-maker's \emph{dynamic regret} over the $K$ rounds of this game is defined as
$$D_K := \displaystyle\sum_{k=1}^K R_k = \displaystyle\sum_{k=1}^K \|\nabla f(\bx_k)\|_{\bI - \bS_k(\bS_k^\top \bS_k)^{-1}\bS_k^\top}, $$
the cumulative sum of the instantaneous regret over all $K$ rounds. 
As with more familiar regret measures, a reasonable and achievable goal in playing Framework~2
is to make a sequence of actions that attains \emph{sublinear dynamic regret}, meaning that~$D_K\in o(K)$. 
The intuition behind seeking a bandit method that guarantees sublinear dynamic regret is that this means, by definition of $o(K)$, that $D_K/K\to 0$ as $K\to\infty$, implying that, on average, the bandit method is long-run regret-free.

\section{Linear UCB Method for Playing Framework~2 }\label{sec:method}
For ease of analysis, we consider an extreme case where the environment in Framework~2 always chooses $p_k=0$, so that $\bS_k$ is empty; it is transparent how one would adapt the algorithm and analysis for the more general case, but bookkeeping and notation become tedious. 
In a practical sense, the environment always choosing $p_k=0$ is tantamount to performing absolutely no sketching via $\bS_k$.
This trivial selection of~$\{p_k\}$ better highlights the fact that the regret bounds we prove concern the satisfaction of~\eqref{eq:rough_desire} and are not actually related to the dynamics of the Framework~1 optimization method employed in Framework~2. 
It is also favorable from a practical point of view that the restriction of $p_k=0$ yields a deterministic method.
\replace{}{As such, we make the critical observation that the results proven in \Cref{lem:gradient_error}, \Cref{lem:potential_lemma} and \Cref{thm:total_regret} are all \emph{deterministic} results given a (trivial) sequence of $\{\bS_k\}$. }

\replace{}{We additionally comment that the linear UCB algorithm that we employ for this nonstationary environment most naturally works for one-dimensional sketches, since one-dimensional sketches correspond to the notion of having linearly parameterized actions. In future work, it may be interesting to attempt to solve a version of the subproblem \eqref{eq:confidence_subproblem} that optimizes over the Stiefel manifold of orthonormal $p_k$-frames for $p_k>1$, but this poses immediate questions about tractability.} 

Pseudocode for our linear UCB method is provided in \Cref{alg:ss_ucb}. \replace{}{We will note before stating \Cref{alg:ss_ucb} that every iteration will require some upper bound $U_k\geq \|\nabla f(\bx_k)\|$. 
As we will demonstrate later in \Cref{cor:no_knowledge}, simply knowing a global upper bound on the gradient norm is sufficient, i.e. we suppose we have knowledge of $G$ in the following assumption.}

\begin{assumption}
\label{ass:bounded_gradient}
The gradient is bounded like $\|\nabla f(\bx)\|\leq G$ for some $G>0$ for all~$\bx\in L(\bx_0)$. 
\end{assumption}

\begin{algorithm2e}
 \SetAlgoNlRelativeSize{-4}

 \textbf{(Inputs)} 
 Initial point $\bx_1$,
 horizon $K$, 
regularizer $\lambda > 0$. 
   
\textbf{(Initializations)} 
Initial covariance matrix $\bC_1(\gets\lambda \bI_d)$,
initial right-hand side $\bb_1(\gets \bzero_d)$, 
initial gradient estimate $\bg_1(\gets \bzero_d)$,
stepsizes $\{\alpha_k\} > 0$,
memory parameter $M\geq 0$. 

\For{$k=1,2,\dots, K$}{
\textbf{(Identify a maximizer of a confidence ellipse around gradient estimate.)}

Determine some upper bound $U_k \geq \|\nabla f(\bx_k)\|$.

Compute
\begin{equation}\label{eq:confidence_subproblem}
    \bs_k\gets \displaystyle\arg\max_{\bs: \|\bs\|=1} \bg_k^\top \bs + \sqrt{\lambda}U_k\|\bs\|_{\bC_k^{-1}}
\end{equation}
\label{line:solve_subproblem}

\textbf{(Compute sketch of gradient.)}
Call oracle to compute 
\begin{equation}
\label{eq:sketch_gradient}
r_k \gets \langle \nabla f(\bx_k), \bs_k\rangle.
\end{equation}

\textbf{(Update gradient estimate.)}

$\bC_{k+1} \gets \lambda \bI_d + \displaystyle\sum_{j=\max\{1, k-M\}}^k \bs_j \bs_j^\top.$ \label{line:covariance_update}

$\bb_{k+1} \gets \displaystyle\sum_{j=\max\{1, k-M\}}^k r_j \bs_j$
\label{line:mean_update}

\begin{equation}
    \label{eq:gradient_estimate}
    \bg_{k+1} \gets \bC_{k+1}^{-1} \bb_{k+1}
\end{equation}

\textbf{(Update incumbent.)}
$\bx_{k+1}\gets \bx_k - \alpha_k r_k \bs_k$. \label{line:update_incumbent}
} 

\caption{Linear Upper Confidence Bound (UCB) Method for Framework~2 \label{alg:ss_ucb}}
\end{algorithm2e}

\Cref{alg:ss_ucb}, and much of our analysis, is based on prior work in \cite{cheung2019learning, russac2019, zhao2020, cheung2022hedging}. 
As previously suggested, \Cref{alg:ss_ucb} dynamically estimates $\nabla f(\bx_k)$ via $\bg_k$ and moreover coarsely quantifies the decision-maker's present uncertainty about $\bs_k^\top \nabla f(\bx_k)$ in recently unexplored directions $\bs_k$ via the regularized covariance matrix $\bC_k$. 
This coarse uncertainty quantification is performed by regularized linear least squares, using the actions $\bs_k$ as linear measurements (maintained in the matrix $\bC_k$) and the directional derivatives $\langle \nabla f(\bx_k),\bs_k \rangle$ as responses (maintained in the vector $\bb_k$). 
We note the presence of a memory parameter $M$, which allows sufficiently old measurement/response pairs to be forgotten; this $M$ defines a so-called sliding window mechanism.
Such forgetting is critical in a nonstationary bandit problem like that described in Framework~2.  
Together, $\bg_k$ and $\bC_{k+1}^{-1}$ define the isocontours of a \emph{confidence ellipse}. 
The width of the confidence ellipse is parameterized by \replace{an estimated}{the} upper bound $U_k$ on~$\|\nabla f(\bx_k)\|$. 

With \Cref{ass:bounded_gradient}, one could take $U_k = G$ for all $k$. 
However, because \Cref{alg:ss_ucb} iteratively selects $\bs_k$ as a maximizer over the confidence ellipse in \Cref{line:solve_subproblem}, one can imagine that tighter upper bounds $U_k$ are generally preferable, so as to encourage better exploitation, that is, better identification of maximizers closer to the mean estimate $\bg_k$. 
We will discuss a practical means of dynamically adjusting $U_k$ in our numerical experiments.
Note, however, that our analysis requires only an upper bound, and so \Cref{ass:bounded_gradient} suffices for that purpose. 

We set out to prove sublinear regret of \Cref{alg:ss_ucb}, but we first require two lemmata, the proofs of which are left to the Appendix. 
\replace{}{The first such lemma bounds the error in the gradient approximation $\bg_k$ to $\nabla f(\bx_k)$ in arbitrary normalized directions~$\bs$ in the $k$th iteration of \Cref{alg:ss_ucb}.}

\begin{lemma}\label{lem:gradient_error}
For each $k=1,2,\dots, K$, we have for any $\bs\in\Reals^d$ satisfying $\|\bs\|=1$, 

$$\|\bs^\top(\nabla f(\bx_k) - \bg_k)\| \leq
\displaystyle\sqrt{\frac{d(M+1)}{\lambda}}
 \displaystyle\sum_{i=\ell(k)}^{k-1} \|\nabla f(\bx_i) - \nabla f(\bx_{i+1})\|
 + 
 \sqrt{\lambda}\|\nabla f(\bx_k)\|\|\bs\|_{\bC_k^{-1}},
$$
where $\ell(k) := \max\{1, k-M-1\}$\replace{.}{, and we interpret the summation to be 0 when $k=1$.} 
\end{lemma}

\replace{}{The second lemma is a technical result needed to prove \Cref{thm:total_regret}; at a high level, \Cref{lem:potential_lemma} quantifies how the inverse covariance matrix $\bC^{-1}$ scales the optimal solution to \eqref{eq:confidence_subproblem} in each iteration of \Cref{alg:ss_ucb}. More precisely, \Cref{lem:potential_lemma} bounds the cumulative sum of these scaled sketches; such results are frequently referred to as \emph{elliptical potential lemmata} in the linear bandit literature.}

\begin{lemma}
\label{lem:potential_lemma}
Let $M\geq 1$, and
denote $\ell(j) = \max\{1, j - M - 1\}$.
Let $\bC_0 = \lambda \bI_d$, and let $$\bC_j = \bC_{j-1} + \bs_j\bs_j\replacemath{}{^\top} - \bs_{\ell(j)}\bs_{\ell(j)}^\top$$ for an arbitrary sequence of unit vectors $\{\bs_j\}_{j=1}^K\subset\Reals^d$. 
Then,
$$\displaystyle\sum_{j=1}^k \|\bs_j\|_{\bC_{j-1}^{-1}}^2
\leq \displaystyle\frac{2kd}{M}\log\left(1 + \displaystyle\frac{M}{\lambda d}\right).
$$
\end{lemma}

We can now prove our main theorem \replace{.}{, bounding the dynamic regret of \Cref{alg:ss_ucb}.}
To make the statement precise, we define one additional piece of notation, the total variation of the sequence of gradients: that is, 
$$V_K := \displaystyle\sum_{k=1}^K \|\nabla f(\bx_{k+1}) - \nabla f(\bx_k)\|.$$

\begin{theorem}\label{thm:total_regret}
Suppose $f$ satisfies \Cref{ass:f}. 
    Given a sequence of gradients~$\{\nabla f(\bx_k)\}$, abbreviate
    $\bs_k^* := \nabla f(\bx_k)/\|\nabla f(\replacemath{\bx}{\bx_k})\|.$
    The cumulative dynamic regret of \Cref{alg:ss_ucb}, with $M\geq 1$, after $K$ iterations is bounded as
    $$D_K \leq \displaystyle\sum_{k=1}^K (\bs_k^* - \bs_k)^\top\nabla f(\bx_k) \leq 2M\displaystyle\sqrt{\frac{d(M+1)}{\lambda}}V_K
+ \displaystyle\sqrt{
\frac{8\lambda d\sum_{k=1}^KU_k^2}{M}\log\left(1 + \displaystyle\frac{M}{\lambda d}\right)}\sqrt{K}.$$
\end{theorem}

\begin{proof}

    For brevity of notation, let 
    $$\beta_k := \displaystyle\sqrt{\frac{d(M+1)}{\lambda}}
 \displaystyle\sum_{i=\ell(k)}^{k-1} \|\nabla f(\bx_{i+1}) - \nabla f(\bx_i)\|.$$
    
    From \Cref{lem:gradient_error}, we have in the $k$th iteration of \Cref{alg:ss_ucb} both
$$\bs_k^{*\top} \nabla f(\bx_k) \leq \bs_k^{*\top} \bg_k +  \beta_k + \sqrt{\lambda}\|\nabla f(\bx_k)\|\|\bs_k^*\|_{\bC_k^{-1}}$$ 
and
$$\bs_k^\top \nabla f(\bx_k) \geq \bs_k^\top \bg_k -  \beta_k - \sqrt{\lambda}\|\nabla f(\bx_k)\|\|\bs_k\|_{\bC_k^{-1}}.$$
Combining these two inequalities, 
$$\begin{array}{rl}
(\bs_k^* - \bs_k)^\top \nabla f(\bx_k)
& \leq
(\bs_k^*-\bs_k)^\top \bg_k + 2\beta_k + \sqrt{\lambda}\|\nabla f(\bx_k)\|\left[\|\bs_k^*\|_{\bC_k^{-1}} + \|\bs_k\|_{\bC_k^{-1}}\right]\\
& \leq 
(\bs_k^*-\bs_k)^\top \bg_k + 2\beta_k + \sqrt{\lambda}U_k \left[\|\bs_k^*\|_{\bC_k^{-1}} + \|\bs_k\|_{\bC_k^{-1}}\right]
\\
& \leq
2\beta_k + 2\sqrt{\lambda}U_k\|\bs_k\|_{\bC_k^{-1}},
\end{array}
$$
where the last inequality is because 
 $\bs_k$ is a maximizer in \cref{eq:confidence_subproblem}.
Thus, 
\begin{equation*}
\begin{array}{rl}
\displaystyle\sum_{k=1}^K (\bs_k^* - \bs_k)^\top\nabla f(\bx_k) & \leq 2\displaystyle\sum_{k=1}^K \beta_k + 2\sqrt{\lambda}\displaystyle\sum_{k=1}^K U_k\|\bs_k\|_{\bC_k^{-1}} \nonumber\\
&= 2\displaystyle\sqrt{\frac{d(M+1)}{\lambda}}\sum_{k=1}^K \sum_{i=\ell(k)}^{k-1} \|\nabla f(\bx_{i+1}) - \nabla f(\bx_i)\| + 2\sqrt{\lambda}\displaystyle\sum_{k=1}^K U_k\|\bs_k\|_{\bC_k^{-1}} \nonumber\\ 
 & \leq 2M\displaystyle\sqrt{\frac{d(M+1)}{\lambda}}V_K + 2\sqrt{\lambda}\displaystyle\sum_{k=1}^K U_k\|\bs_k\|_{\bC_k^{-1}} \\
& \leq 2M\displaystyle\sqrt{\frac{d(M+1)}{\lambda}}V_K + 2\sqrt{\lambda}\displaystyle\sqrt{\frac{2Kd}{M}\log\left(1 + \displaystyle\frac{M}{\lambda d}\right)}\sqrt{\sum_{k=1}^K U_k^2},\\
\end{array}
\end{equation*}
where we have used \Cref{lem:potential_lemma} and the Cauchy--\replace{Schwartz}{Schwarz} inequality to derive the last inequality.
\end{proof}

\subsection{Understanding \Cref{thm:total_regret}}
To better understand \Cref{thm:total_regret}, 
we will focus on \replace{the simple iteration (3)}{\Cref{alg:linesearch}.} 
\replace{Suppose $\alpha_k = \frac{1}{L}$ for all $k$.}{}
We first note that, as a very loose bound, 
$$\|\nabla f(\bx_{k+1}) - \nabla f(\bx_k)\| \leq L \|\bx_{k+1} - \bx_k\| = L\|\replacemath{\frac{1}{L}}{\alpha_k} \bP_k \nabla f(\bx_k)\| \leq \replacemath{}{L}\|\nabla f(\bx_k)\|$$
for all $k=1,\dots,K$ due to \Cref{ass:f}, and so
\begin{equation}\label{eq:vbound}
V_K \leq \replacemath{}{L}\displaystyle\sum_{k=1}^K \|\nabla f(\bx_k)\|.
\end{equation}
We stress that \cref{eq:vbound} is practically quite loose; for many functions satisfying \Cref{ass:f}, the bound $\|\nabla f(\bx_{k+1}) - \nabla f(\bx_k)\| \leq L\|\bx_{k+1}-\bx_k\|$ is not usually tight, given that $L$ is a \emph{global} Lipschitz constant per \Cref{ass:f}. 
Nonetheless, we will use the coarse upper bound in \cref{eq:vbound} to derive \Cref{cor:no_knowledge} and \Cref{cor:perfect_knowledge}. 
Moreover, we can further upper bound the right-hand side in \eqref{eq:vbound} by using a straightforward result. 

\begin{theorem}\label{thm:opt_randomized}
    Let $K>0$, and let $\delta^' \replacemath{> 0}{\in (0,1)}$. 
    Let \Cref{ass:sketch} hold, choosing $\delta = \delta^'/K$, and letting $\replacemath{q>0}{q\in(0,1)}$ be arbitrary. 
    Let \Cref{ass:sketch2} hold. 
    Then, \replace{the iterative method defined by (3)}{\Cref{alg:linesearch}} satisfies, with probability at least $1-\delta^'$, 
    \begin{equation}
\label{eq:sum_squares}\displaystyle\sum_{k=0}^{K-1} \|\nabla f(\bx_k)\|^2 \leq 
\displaystyle\frac{2S_{\max}L(f(\bx_0) - f(\bx_*))}{q\replacemath{}{(2\beta-1)}}.
\end{equation}
    Moreover, as a simple consequence of Cauchy--Schwarz inequality, 
    \begin{equation}\label{eq:sum}
    \displaystyle\sum_{k=0}^{K-1} \|\nabla f(\bx_k)\| \leq 
    \displaystyle\sqrt{\displaystyle\frac{2S_{\max}L(f(\bx_0) - f(\bx_*))}{q\replacemath{}{(2\beta-1)}}}\sqrt{K}.
    \end{equation} 
\end{theorem}

\begin{proof}
    By \replace{Taylor's theorem}{\Cref{lem:bounds_on_decrease}}, we have for all $k\leq K-1$
    $$
    \begin{array}{rl}
    f(\bx_{k+1}) & \replacemath{\leq f(x_k) -\frac{1}{L} 
   \nabla f(x_k)^\top S_k (S_k^\top S_k)^{-1} S_k^\top \nabla f(x_k) + \frac{1}{2L}\|S_k (S_k^\top S_k)^{-1} S_k^\top \nabla f(x_k)\|^2}{}\\
   & \replacemath{= f(x_k) - \frac{1}{L}\|S_k^\top \nabla f(x_k)\|^2_{(S_k^\top S_k)^{-1}}
   + 
   \frac{1}{2L}\|S_k^\top \nabla f(x_k)\|^2_{(S_k^\top S_k)^{-1}}}{}\\
   & = f(\bx_k) - \frac{\replacemath{1}{2\beta-1}}{2L}\|\bS_k^\top \nabla f(\bx_k)\|^2_{(\bS_k^\top \bS_k)^{-1}}\\
   & \leq f(\bx_k) - \frac{\replacemath{1}{2\beta-1}}{2S_{\max}L}\|\bS_k^\top \nabla f(\bx_k)\|^2,\\
   \end{array}$$
   where \Cref{ass:sketch2} was used to derive the last inequality. 
   By \Cref{ass:sketch}, with probability \replace{}{at least} $1-(\delta^'/K)$, 
   $$f(\bx_{k+1}) \leq f(\bx_k) - \frac{q\replacemath{}{(2\beta-1)}}{2S_{\max}L}\|\nabla f(\bx_k)\|^2$$
   Applying an union bound, we have
   $$f(\bx_0) - f(\bx_K) \geq f(\bx_0) - f(\bx_K) \geq \displaystyle\frac{q\replacemath{}{(2\beta-1)}}{2S_{\max}L}\displaystyle\sum_{k=0}^{K-1} \|\nabla f(\bx_k)\|^2 \quad \text{w.p. } 1-\delta^'.$$
\end{proof}

The use of an union bound in the proof of \Cref{thm:opt_randomized} is admittedly brutish. As a result,  for very large horizons $K$, we require the probability $1-(\delta^'/K)$ in \Cref{ass:sketch} to be so close to 1 that regardless of how $\bS_k$ is generated, $p_k\approx d$. 
Moreover, \Cref{thm:opt_randomized} makes sense only when $K$ is finite, which is reasonable for our setting but would be generally nonsensical for performing rate analyses. 
One route to conducting a more refined analysis is to consider imposing a globalization strategy so that the high-probability convergence rate analysis of \cite{cartis2022randomised} could be imported here. 
Attempting this analysis would complicate our presentation, however, and distract from the purpose of this manuscript. 
Instead, we interpret \Cref{thm:opt_randomized} as telling us that, so long as an optimization algorithm used within Framework~2 yields similar first-order optimality guarantees to \replace{subspace gradient descent method in (3)}{\Cref{alg:linesearch}}, an upper bound like $V_k\in \mathcal{O}(\sqrt{K})$ is plausible. 

\replace{}{We additionally comment here that the probabilistic nature of \Cref{thm:opt_randomized} is due entirely to the probabilistic satisfaction of \Cref{ass:sketch}. 
In particular, had we instead made the stronger assumption that \eqref{eq:well-aligned} is satisfied surely for all $k=1,\dots, K$, then \Cref{thm:opt_randomized} would also hold surely, since the UCB algorithm behaves deterministically as a function of the input sequence $\{\bS_k\}$.}

We can now state our two corollaries. 

\begin{corollary}\label{cor:no_knowledge}
    Let \Cref{ass:bounded_gradient} hold. 
    Let the setting and assumptions of \Cref{thm:opt_randomized} hold. 
    Suppose \replace{the iteration (3)}{\Cref{alg:linesearch}} is employed as the optimization algorithm in Framework~2\replace{with a constant stepsize $\alpha_k = \frac{1}{L}$}{}.
    Then, with probability at least $1-\delta^'$, 
    $$D_K \leq 
    \left(2M\displaystyle\sqrt{\frac{2d (M+1) S_{\max} L (f(\bx_0) - f(\bx_*))}{q\replacemath{}{(2\beta-1)} \lambda}}     
+ \displaystyle\sqrt{
\frac{8\lambda d KG^2}{M}\log\left(1 + \displaystyle\frac{M}{\lambda d}\right)}\right)\sqrt{K},$$
and so if we choose $M = \lceil K^{\frac{1}{4}}\rceil$\replace{}{, where $\lceil\cdot\rceil$ denotes the ceiling function}, then 
$$D_K\in\mathcal{O}(K^{\frac{7}{8}}). $$
\end{corollary}

Additionally, we can prove a result for the idealized case where $U_k = \|\nabla f(\bx_k)\|$ (that is, the upper bound used in \cref{eq:confidence_subproblem} is maximally tight) on every iteration.

\begin{corollary}\label{cor:perfect_knowledge}
    Let the setting and assumptions of \Cref{thm:opt_randomized} hold. 
    Suppose \replace{the iteration (3)}{\Cref{alg:linesearch}} is employed as the optimization algorithm in Framework~2\replace{ with a constant stepsize $\alpha_k = \frac{1}{L}$}{}. 
    Suppose that $U_k = \replacemath{}{\|}\nabla f(\bx_k)\replacemath{}{\|}$ on every iteration
    of \Cref{alg:ss_ucb}. 
    Then, with probability at least $1-\delta^'$,
    $D_K \leq C\sqrt{K}$, where 
    $$C :=  2M\displaystyle\sqrt{\frac{2d S_{\max} L (f(\bx_0) - f(\bx_*))(M+1)}{q\replacemath{}{(2\beta-1)} \lambda}}     
+ \displaystyle\sqrt{
\frac{16\lambda d S_{\max} L (f(\bx_0) - f(\bx_*))}{q\replacemath{}{(2\beta-1)} M}\log\left(1 + \displaystyle\frac{M}{\lambda d}\right)}$$
and so if $M\in \mathcal{O}(1)$, then 
$$D_K\in \mathcal{O}(\sqrt{K}).$$
\end{corollary}

From \Cref{cor:no_knowledge} and \Cref{cor:perfect_knowledge}, we see that, unsurprisingly, the regret bound is better when we have perfect knowledge of the magnitude $\|\nabla f(\bx_k)\|$ on every iteration of \Cref{alg:ss_ucb}. 
This knowledge of $\|\nabla f(\bx_k)\|$ obviously renders the algorithm analyzed in \Cref{cor:perfect_knowledge} impractical. 
Moreover, even the algorithm analyzed in \Cref{cor:no_knowledge} is impractical, although considerably less so; the algorithm considered in \Cref{cor:no_knowledge} requires knowledge of 
the gradient upper bound $G$ in \Cref{ass:bounded_gradient}.
Regardless, \Cref{cor:no_knowledge} and \Cref{cor:perfect_knowledge} still suggest that, given a subspace optimization algorithm in Framework~1 that can yield  upper bounds on $V_K$ similar to or \replace{betterthan}{better than} the coarse bound from \Cref{thm:opt_randomized}, the linear UCB mechanism in \Cref{alg:ss_ucb} will exhibit sublinear dynamic regret, as we intended to demonstrate.

\section{Practical Considerations}\label{sec:considerations}
Motivated by the impracticalities discussed at the end of the preceding section, we consider a few alterations to \Cref{alg:ss_ucb} to yield a more practical algorithm, stated in pseudocode in \Cref{alg:practical}. 
We now discuss the differences between \Cref{alg:ss_ucb} and \Cref{alg:practical}. 

\begin{algorithm2e}[H]
 \SetAlgoNlRelativeSize{-4}

 \textbf{(Inputs)} 
 Initial point $\bx_1$,
 horizon $K$, 
regularizer $\lambda > 0$. 
   
\textbf{(Initializations)} 
Initial covariance matrix $\bC_1(\gets\lambda \bI_d)$,
initial right-hand side $\bb_1(\gets \bzero_d)$, 
initial gradient estimate $\bg_1(\gets \bzero_d)$,
memory parameter $M\geq 0$,
average weighting parameter $\mu\in(0,1)$,
backtracking parameter $\beta\in(0,1)$, 
sufficient decrease parameter $\sigma > 0$,
initial gradient upper bound estimate $U_1$. 

\For{$k=1,2,\dots, K$}{
\textbf{(Obtain (random) sketch)} 

Receive $\bS_k \in\Reals^{d\times p_k}$ and call oracle to compute $\bS_k^\top \nabla f(\bx_k)$. \label{line:call_oracle1}

\textbf{(Compute gradient upper bound estimate)}\label{line:first}

\eIf{$k=1$}
{$U_1\gets\displaystyle\frac{d}{p_1}\|\bS_1^\top\nabla f(\bx_1)\|$\replace{}{.}}
{$U_k \gets \mu U_{k-1} + (1-\mu)\displaystyle\frac{d}{p_k}\|\bS_k^\top\nabla f(\bx_k)\|$\replace{}{.} \label{line:update_estimate}}

\textbf{(Identify a maximizer of a confidence ellipse around gradient estimate.)}

Compute
$\bs_k\gets \displaystyle\arg\max_{\bs: \|\bs\|=1} \bg_k^\top \bs + \sqrt{\lambda}U_k\|\bs\|_{\bC_k^{-1}}$\replace{}{.}

\textbf{(Compute one-dimensional sketch of gradient.)}

Call oracle to compute 
$r_k \gets \langle \nabla f(\bx_k), \bs_k\rangle.$ \label{line:call_oracle2}

\textbf{(Update gradient estimate.)}

Compute $\bC_{k+1}$ and $\bb_{k+1}$ according to \eqref{eq:new_mean_update}. 

$\bg_{k+1} \gets \bC_{k+1}^{-1} \bb_{k+1}$\replace{}{.}\label{line:last}


\textbf{(Determine next incumbent by backtracking line search.)}\label{line:stepsize_adjust}

$\alpha\gets 1$\replace{}{.}

$\bx_{+}\gets \bx_k - \alpha \bP_k \nabla f(\bx_k)$.

\While{$f(\bx_{+}) \geq f(\bx_k) - \sigma \alpha_k\|\bP_k\nabla f(\bx_k)\|^2$}
{
$\alpha\gets\beta\alpha.$

$\bx_{+}\gets \bx_k - \alpha \bP_k \nabla f(\bx_k)$.
}

$\bx_{k+1}\gets\bx_{+}$. 
} 

\caption{Practical Variant of Sliding Window UCB Method \label{alg:practical}}
\end{algorithm2e}

\subsection{Incorporating arbitrary sketches}\label{sec:arbitrary_sketches}
\Cref{cor:no_knowledge} and \Cref{cor:perfect_knowledge} demonstrated how, in order to bound $V_K$ (and hence $D_K$) meaningfully, 
the linear UCB mechanism might be used to \emph{augment}, as opposed to \emph{replace}, randomized sketches. 
Thus, it is natural to extend \Cref{alg:ss_ucb} so that the memory parameter $M$ continues to specify a number of past iterations that we wish to keep in memory, but the mean update and covariance update in \Cref{line:mean_update} and \Cref{line:covariance_update}, respectively, use all the data acquired in a sketch.

In particular, 
we suppose that given a sequence of sketch sizes $\{p_k\}_{k=1}^K$, we acquire a sequence of sketch matrices $\{\bS_k\in \Reals^{d\times p_k}\}$. 
After $\bs_k$ is computed in \Cref{line:solve_subproblem}, we augment~$\bS_k$ by the column $\bs_k$ so that $\bS_k \in \Reals^{d\times (p_k + 1)}$. 
We then add a second subscript to the variables $r$ and $\bs$ so that, for each iteration $k$, $\{r_{k,i}: i=1,\dots,p_k, p_k + 1\}$
corresponds to the entries of $\bS_k^\top\nabla f(\bx_k)$ and 
and $\{\bs_{k_i}: i = 1,\dots, p_k, p_k+1\}$ corresponds to the columns of $\bS_k$. 
We then replace \Cref{line:covariance_update} and \Cref{line:mean_update}, respectively, with 
\begin{equation}\label{eq:new_mean_update}
\bC_{k+1} \gets \lambda \bI_d + \displaystyle\sum_{j=\max\{1,k-M\}}^k
\displaystyle\sum_{i=1}^{p_j+1} \bs_{j,i} \bs_{j,i}^\top 
\quad \text{ and } \quad
\bb_{k+1} \gets \displaystyle\sum_{j=\max\{1,k-M\}}^k
\displaystyle\sum_{i=1}^{p_j+1} r_{j,i} \bs_{j,i}. 
\end{equation}
Then, we replace the incumbent update in \Cref{line:update_incumbent} with 
$$\bx_{k+1} \gets \bx_k - \alpha_k \bP_k \nabla f(\bx_k), $$
as in \replace{(3)}{\Cref{alg:linesearch}}, noting that this incumbent update requires the additional (implicit) computation of $\bS_k (\bS_k^\top \bS_k)^{-1}$.

\subsection{Estimating $\|\nabla f(\bx_k)\|$}
We saw in comparing \Cref{cor:perfect_knowledge} to \Cref{cor:no_knowledge} that having knowledge of $\|\nabla f(\bx_k)\|$ in each iteration $k$ for use in the subproblem \eqref{eq:confidence_subproblem} leads to better bounds on dynamic regret. 
While we generally cannot know $\|\nabla f(\bx_k)\|$ in practice, we have found it is beneficial to approximate $\|\nabla f(\bx_k)\|$ \replace{online}{dynamically}. 
We accomplish this by a simple exponential moving average applied to the quantity $\|(\bS_k^T\bS_k)^{-1}\bS_k^\top \nabla f(\bx_k)\|$ computed in each iteration $k$; see \Cref{line:update_estimate} of \Cref{alg:practical}.

\subsection{Memory management}
Throughout \Cref{alg:practical} we manage memory by maintaining a dictionary of size $M$ of sketches $\bS_k$ and their corresponding linear measurements $r_k = \bS_k^\top\nabla f(\bx_k)$. 
In the prototype implementation that we tested, the action of the inverse covariance matrix is performed by accessing the dictionary and iteratively applying the Sherman--Morrison--Woodbury formula\replace{}{; see, e.g.,~\cite[Equation 2]{hagerupdating1989}}. 

For the solution of \eqref{eq:confidence_subproblem}, that is, the maximization of the confidence ellipse, we implemented a basic projected gradient descent method, motivated by the fact that projection onto the 2-norm constraint is trivial. 
The most computationally expensive part of computing the gradient of the objective function of \eqref{eq:confidence_subproblem} amounts to the matrix-vector multiplication with the inverse covariance matrix, again via the Sherman--Morrison--Woodbury formula.
Because the norm term in the objective of \eqref{eq:confidence_subproblem} is nondifferentiable at the origin, we generate a random initial starting point of norm $0.01$.

\subsection{A subspace variant of \texttt{POUNDers}}
In addition to implementing \Cref{alg:practical} as described, 
we  applied the UCB mechanism discussed and analyzed in this manuscript to the DFO solver \texttt{POUNDers} \cite{SWCHAP14} to yield a variant of \texttt{POUNDers} that fits within the subspace optimization Framework~1. \replace{}{Our motivation for introducing this variant is to demonstrate that the application of the linear bandit method of \Cref{alg:ss_ucb} to any method in Framework~1 is relatively straight-forward, and need not only apply to the fairly simple \Cref{alg:linesearch}.}

\texttt{POUNDers} is a model-based DFO method that is actively maintained in \texttt{IBCDFO} \cite{osti_2382683}. 
While \texttt{POUNDers} in its default setting is designed for nonlinear least-squares minimization, we will present \texttt{POUNDers} here as a method intended to solve \eqref{eq:minf}; extending this methodology to nonlinear least squares (and other compositions of an inexpensive smooth function with an expensive black box) is fairly straightforward.

As a model-based method, \texttt{POUNDers} maintains a history of previously evaluated points and their corresponding function evaluations,~$\mathcal{Y}=~\{((\by^1, f(\by^1)),\dots,(\by^{|\mathcal{Y}|}, f(\by^{|\mathcal{Y}|}))\}$.
\texttt{POUNDers} is, moreover, a \emph{model-based trust-region method}.
Common to all model-based trust-region DFO methods (see, e.g., \replace{\cite{LMW2019AN}[Section 2.2]}{\cite[Section 2.2]{LMW2019AN}}), every iteration of \texttt{POUNDers} has both an incumbent~$\bx_k$ and a trust-region radius $\Delta_k$. 
On every iteration, \texttt{POUNDers} applies \Cref{alg:subspace} to identify, given $\bx_k$, $\Delta_k$, and $\mathcal{Y}$, a set of mutually orthonormal columns~$\bS$ and a perpendicular set of mutually orthonormal columns, $\bS^\perp$, so that a subset of points in $\mathcal{Y}$ both (1) are sufficiently close (measured as a factor of $\Delta)$ to $\bx_k$ and (2) exhibit sufficient volume in the affine subspace defined by $\bx_k$ and $\bS$ (as measured by projection onto the perpendicular subspace). 
\Cref{alg:subspace} is a greedy algorithm that grows the set of columns of $\bS$, one at a time, and is implemented by performing QR insertions. 

Having computed $\bS$ and $\bS^\perp$, \texttt{POUNDers} will then perform the function evaluations~$\{f(\bx_k+\Delta_k \bp): \bp \text{ is a column of } \bS^\perp\}$, provided $\bS^\perp$ is nonempty; the points~$\{\bx_k+~\Delta_k \bp: \bp \text{ is a column of } \bS^\perp\}$  are referred to as \emph{geometry points}. 
With these geometry points now included in $\mathcal{Y}$, one can prove that there exists a subset of $\mathcal{Y}$ so that a quadratic model interpolating those points locally (as measured by~$\Delta_k$) exhibits errors of the same order as a first-order Taylor model; see, for example,~\replace{\cite{SW08}[Theorem 4.1]}{\cite[Theorem 4.1]{SW08}}. 

In this aspect of \texttt{POUNDers}  we make a large modification to yield ``subspace \texttt{POUNDers}," or \texttt{SS-POUNDers} for short. 
Rather than evaluating a stencil in $\bS^\perp$ to guarantee local accuracy of a model of $f$ in the full space $\Reals^d$ , 
\texttt{SS-POUNDers} employs a confidence ellipse $\bC_k$ and mean estimate of a gradient $\bg_k$, as in \Cref{alg:ss_ucb}, and it selects a single column $\bs_k$ of $\bS^\perp$ by solving

\begin{equation}
    \label{eq:ss_pounders_ucb}
    \bs_k \gets \arg\max_{\bs \in \bS^\perp} \bg_k^\top \bs + \sqrt{\lambda} U_k \|\bs\|_{\bC_k^{-1}}. 
\end{equation}

In \texttt{SS-POUNDers}, we then augment $\bS$ from \Cref{alg:subspace} by the column $\bs_k$ to yield~$\bS_k$.
We can then augment $\bS_k$  by additional randomized sketches in the span of $\bS^\perp$; this additional randomization is performed identically to \replace{\cite{cartis2023scalable}[Algorithm 5]}{\cite[Algorithm 5]{cartis2023scalable}}, and we provide pseudocode in \Cref{alg:random_augment}. 
We then construct a quadratic interpolation model $m_k:\Reals^{\dim(\span(\bS_k))}\to\Reals$ using only points in the affine subspace defined by $\bx_k$ and $\bS_k$;
this is accomplished by using the same model-building techniques as in \texttt{POUNDers} but by fitting the quadratic model to projected points $\bS_k^\top(\by - \bx_k)$ for the selected subset of points $\by\in\mathcal{Y}$. 
We then minimize this lower-dimensional model over a trust-region to obtain an (approximate) minimizer $\bz_k\in\Reals^{\dim(\span(\bS_k))}$, which is embedded back into the full space via the transformation $\bx_k + \bS_k \bz_k$. 
This idea of minimizing a derivative-free model only over a selected affine subspace and then re-embedding the solution in the full space is the same idea employed in, for instance,~\cite{cartis2023scalable}. 
\texttt{SS-POUNDers} updates the estimators $\bC_k$ and $\bb_k$ using both $\bS_k$ and the subspace derivative-free model gradient. 
The rest of \texttt{SS-POUNDers} appears essentially the same as \texttt{POUNDers}, and pseudocode is provided in \Cref{alg:ss_pounders}. 

\begin{algorithm2e}[h!]
\caption{Identify Initial Subspace}
    \label{alg:subspace}
    \textbf{Input:} Center point $\bx_k\in\Reals^d$, bank of evaluated points $\cY = \{(\by^1,f(\by^1),\dots,(\by^{|\cY|},f(\by^{|\cY|}))\}$ satisfying $\replacemath{(\bx, f(\bx))}{(\bx_k, f(\bx_k))\in\cY}$, trust-region radius $\Delta_k$. \\
    \textbf{Initalize: } Choose algorithmic constants $c\geq 1$, $\theta_1\in(0,\frac{1}{c}]$.\\
    Set $\bS = \{\bs^1\} = \{\bzero_d\}$. \\
    Set $\bS^\perp = \bI_d$. \\
    \For{$i=1,\dots,|\cY|$}{
        \If{$\|\by^i-\bx\|\leq c\Delta_k$ and $\left|\text{proj}_{\bS^\perp}\left(\frac{1}{c\Delta_k}(\by^i-\bx)\right)\right|\geq\theta_1$}
        {
            $\bS = \bS \cup \{\by^i-\bx\}$\replace{}{.}\\
            Update $\bS^\perp$ to be an orthonormal basis for $\replacemath{\mathcal{N}}{\Null}([\bs^2 \cdots \bs^{|S|}])$\replace{}{.}\\
        }
        \If{$|\bS| = n + 1$}{
        \textbf{break} (the for loop)\replace{}{.}
        }
    }
    $\bS = [\bs^2, \cdots, \bs^{|\bS|}]$\replace{}{.}\\
    $\bQ = [\bS \hspace{0.5pc} \bS^\perp]$\replace{}{.}\\
    \textbf{Return: } $\bS, \bS^\perp, \bQ$ \\
\end{algorithm2e}

\begin{algorithm2e}[h!]
\caption{A Subspace Variant of \texttt{POUNDers}}
\label{alg:ss_pounders}
\textbf{(Initialization)} Choose algorithmic constants $\eta_1, \eta_2, \Delta_{\max} > 0$
and $0 < \nu_1 < 1 <\nu_2$. \\
Choose horizon $K$ and regularizer $\lambda > 0$. 
Choose initial point $\bx^0\in\Reals^d$ and initial trust-region radius $\Delta_0 \in (0, \Delta_{\max})$. \\
Initialize $\bC_1\gets \lambda \bI_d$, $\bb_1 \gets \bzero\in\Reals^d$, gradient estimate $\bg_1\gets \bzero\in\Reals^d$, memory parameter $M$, average weighting parameter $\mu\in(0,1)$, initial gradient upper bound estimate $U_1$. 
Initialize a bank of points $\cY$ with pairs $(\bx,f(\bx))$ for which $f(\bx)$ is known.\\
\For{$k=1,2,\dots$}
{
\textbf{(Get initial subspace)} Use \Cref{alg:subspace} to obtain $\bS_k$, $\bS_k^\perp$ and $\bQ_k$.\\
\eIf{$\bS_k^\perp \neq \emptyset$}{
\textbf{(Determine one-dimensional sketch from UCB)} Find $\bs_k$ by solving \eqref{eq:ss_pounders_ucb}, $\bS_k\gets [\bS_k, \bs_k]$. \label{line:ucb_part} \\
}
{$\bs_k\gets\emptyset.$}
\textbf{(Obtain additional random sketches)} Choose $p_k\in[0, d - \dim(\span(\bS_k))]$ and call \Cref{alg:random_augment} to obtain $\bQ_k\in\Reals^{d\times p_k}$, $\bS_k\gets [\bS_k, \bQ_k]$.   \label{line:choices}\\
\textbf{(Perform additional function evaluations)} Evaluate $\{f(\bx^k + \Delta_k \bq_i): \bq_i \in \{\bs_k\} \cup \col(\bQ_k)\footnotemark \}$ and update $\cY$. \\
\textbf{(Build model)} Build a model $m_k:\Reals^{\dim(\span(\bS_k))} \to \Reals$ of $f$ valid on the affine subspace defined by $\bx_k$ and $\bS_k$. \label{line:build_model} \\
\textbf{(Solve TRSP)} (Approximately) solve $\displaystyle\min_{\bz: \|\bz\|\leq \Delta_k} m_k(\bz)$ to obtain $\bz_k$. \label{line:trsp}\\
\textbf{(Evaluate new point)} Evaluate $f(\bx_k + \bS_k \bz_k)$ and update $\cY$. \\
\textbf{(Determine acceptance)} Compute $\rho_k\gets \displaystyle\frac{f(\bx_k) - f(\bx_k+\bS_k \bz_k)}{m_k(\bzero) - m_k(\bz^k)}$. \\
 \eIf{$\rho_k\geq\eta_1$}{
$\bx_{k+1}\gets \bx_k + \bS_k \bz_k$.}
{$\bx_{k+1}\gets \bx_k$.}
\textbf{(Trust-region adjustment)} \eIf{$\rho_k\geq\eta_1$}{
\eIf{$\|\nabla m_k(\bzero)\|\geq \eta_2\Delta_k$}{
$\Delta_{k+1}\gets \min\{\nu_2\Delta_k,\Delta_{\max}\}.$}
{$\Delta_{k+1} \gets \nu_1\Delta_k$\replace{}{.}}
}
{
$\Delta_{k+1}\gets\nu_1\Delta_k$.
}
\textbf{(Update UCB estimators)} Update $\bC_{k+1}$ and $\bb_{k+1}$ according to \eqref{eq:new_mean_update}, replacing $r_{j,i}$ with corresponding entries of $\nabla m_j(\bzero)$. \\
$\bg_{k+1}\gets \bC_{k+1}^{-1}\bb_{k+1}$.\\
$U_{k+1} \gets \mu U_k + (1-\mu) \frac{d}{\dim(\span(\bS_k))}\|\nabla m_k(\bzero)\|.$\\
}
\end{algorithm2e} 

\begin{algorithm2e}
\caption{Generating orthogonal random sketches in $\bS^\perp$}
    \label{alg:random_augment}
    \textbf{Input:} Orthogonal basis for subspace $\bS$, number of new directions $p_k$. \\
    Randomly generate $\bA\in\Reals^{d\times p_k}$ from some distribution. \\
    $\bA\gets \bA - \bS\bS^\top \bA$\\
    Perform the QR factorization of $\bA$, and return the first $p_k$ columns of the orthonormal $\bQ$ factor. 
\end{algorithm2e}

\section{Numerical Tests}\label{sec:experiments}

The hypothesis motivating the experiments performed in this section is that, without any sophisticated hyperparameter tuning, employing a sliding window UCB mechanism in addition to a standard randomized generation of $\bS_k$ generally leads to practical gains over a method that  performs only randomized sketches.

The work in this manuscript was motivated primarily by applications where gradients are computationally expensive to compute. Thus we implicitly assume that all linear algebra costs incurred in updating the (inverse) covariance matrix and the vector~$\bb$, as well as the costs of performing multiplications with the inverse covariance matrix in solving \eqref{eq:confidence_subproblem}, are dominated by the cost of computing a (sketched) gradient. 
Thus, to model this sort of computational setting, and for the sake of these experiments, we are  interested solely in maximizing objective function decrease given a fixed budget of oracle accesses. 

In our experiments, and motivated by our discussion of when subspace optimization methods within Framework~1 are most appropriate, 
we implemented and tested reasonably faithful versions of both \Cref{alg:practical} and \Cref{alg:ss_pounders}. 
\Cref{alg:practical}, in both \Cref{line:call_oracle1} and \Cref{line:call_oracle2}, assumes access to an implementation of forward-mode AD to return arbitrary directional derivatives.
For easy reproducibility, however, we chose in our experiments to use a large collection of unconstrained problems from CUTEst \cite{Gould2014}.
Because gradients  in CUTEst are computed analytically and are relatively inexpensive to obtain by design, we simulate a forward AD oracle by computing the full gradient $\nabla f(\bx_k)$,
\footnotetext{$\col(\cdot)$ denotes the operation on matrices that maps a matrix to the set of its columns.}
but only allowing the solver to observe the sketched quantity $\bS_k^\top \nabla f(\bx_k)$, given the solver's choice of $\bS_k$. 
\texttt{SS-POUNDers}, our implementation of \Cref{alg:ss_pounders}, is derivative-free. 

\subsection{Testing \Cref{alg:practical}}\label{sec:testing_sdm}
With our hypothesis and computational setting fixed, we consider two versions of \Cref{alg:practical}.
Both methods have access to an oracle that will compute a sketch of size~$p$. 
The first version of the method, which we refer to simply as \emph{random-only}, is purely randomized; in the $k$th iteration, the random-only method generates $\bS_k\in\Reals^{d\times p}$ as a Gaussian matrix. 
\Cref{line:first} through \Cref{line:last} of \Cref{alg:practical}, that is, all the lines that pertain to generating an upper confidence bound on the gradient, are omitted entirely in the random-only method.  
The second method, which we refer to as \emph{UCB}, does not omit \Cref{line:first} through \Cref{line:last} of \Cref{alg:practical}.
In order to give maximally equal oracle access to the two methods, the same sequence of random Gaussian matrices~$\bS_k$ that the random-only method employs is also given to the UCB method but with the last column of $\bS_k$ removed.  
Thus, both methods receive the same number of one-dimensional sketches in every iteration from the oracle, and the randomized directional information they receive is comparable in that, in the $k$th iteration, $p-1$ of the directional derivatives computed are equivalent across both methods; the difference is that the random-only method effectively chooses a random Gaussian vector as its $p$th direction, while the UCB method solves \cref{eq:confidence_subproblem} to determine its $p$th direction. 

We tested the random-only and UCB variants of \Cref{alg:practical} on every unconstrained problem in the CUTEst \cite{Gould2014} test set with problem dimension less than or equal to~10,000. 
A mild downselection was performed in that any problem that produced a \texttt{NaN} or overflow value over the course of any of our tests was removed from consideration. 

We compare the random-only and UCB variants using a relative performance measure. 
In particular, 
for a given problem, we run both variants for a fixed number of seeds (in these experiments, 10) with a fixed budget of directional-derivative oracle calls on a problem set $\mathcal{P}$.
Because the sequence of incumbent values generated by \Cref{alg:practical} is monotone decreasing, we consider the ratios
\begin{equation}\label{eq:ratio_def}
r_{i,p} := \displaystyle\frac{f(\bx^{rand,i,p}) - f(\bx^{ucb,i,p})}{\max\{f(\bx_0) - f(\bx^{rand,i,p}), f(\bx_0) - f(\bx^{ucb,i,p}), 1\}} \in [-1,1],
\end{equation}
where $i$ indexes the 30 seeds, $p$ indexes the problem set $\mathcal{P}$, $\bx^{rand,i,p}$ denotes the last incumbent of the random-only variant on the $i$th seed applied to problem $p$,  $\bx^{ucb,i,p}$ denotes the last incumbent of the UCB variant on the $i$th seed applied to problem $p$, and $\bx_0$ denotes a common initial point. 
Each box-and-whisker plot in our figures corresponds to a single problem $p$ and illustrates the 30 values of $r_{i,p}$. 
From the definition of~$r_{i,p}$ in \eqref{eq:ratio_def}, we see that the UCB variant performs better at the minimization task  \eqref{eq:minf} on a single run with a fixed seed whenever~$r_{i,p} > 0$. 
We separate our problems by problem size, with the smallest problems (dimension $11\leq d\leq 100$) in \Cref{fig:alg2_100}, slightly larger problems (dimension $101\leq d \leq 1000$) in \Cref{fig:alg2_1000}, and 
the largest problems (dimension $1001\leq d\leq 10000$) in \Cref{fig:alg2_10000}. 
In all our tests, the horizon was chosen as $K=1000$. For a high-dimensional problem with $d=10000$ and $p_k = \lceil 0.001 d \rceil$, this amounts to a total of $10,000$ directional derivatives, that is, the same amount of effort that would be required to compute a single full-space gradient. 
Thus, these tests are meant to simulate a true computationally expensive setting that would preclude the use of full gradients.
Our choice to employ relative performance ratios \eqref{eq:ratio_def} in these tests is thus justified, since they represent relative performance within such a constrained budget. 

We stress that we did not strenuously tune any hyperparameters in our implementation of \Cref{alg:practical} in these experiments. 
Some preliminary experiments helped inform our decisions, but all the runs illustrated here used a regularization parameter~$\lambda = 1/d$, a fixed sketch size $p_k = p$ illustrated in the captions of our figures, memory parameter $M = d / p$, averaging parameter $\mu = 0.8$, backtracking parameter $\beta = 0.5$, and sufficient decrease parameter $\sigma = 10^{-8}$. 
\begin{figure}  \includegraphics[width=.9\linewidth]{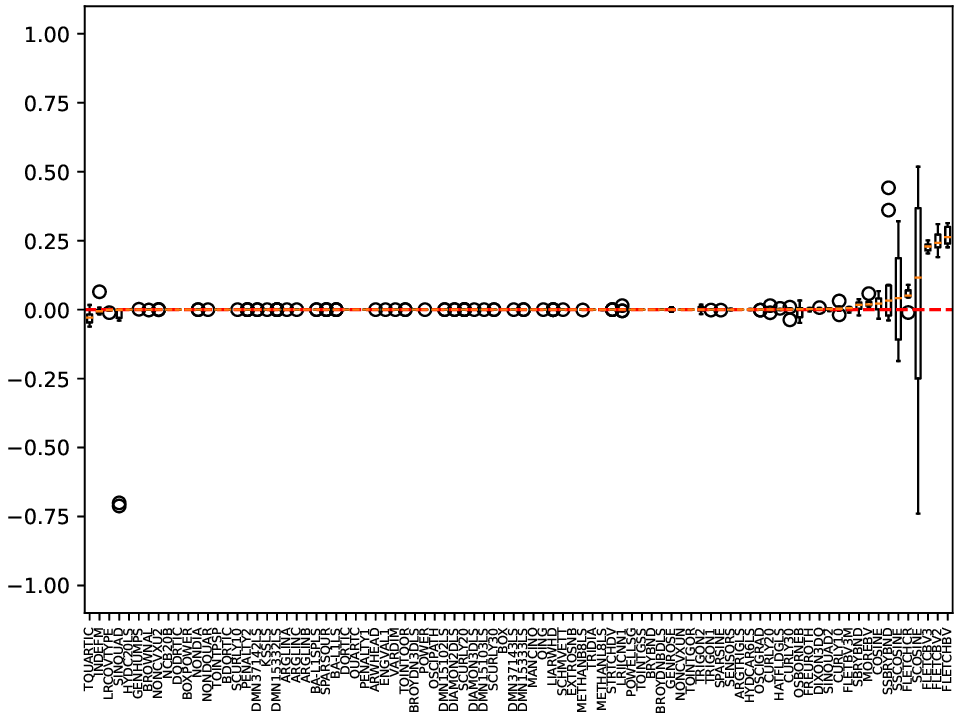}
      \includegraphics[width=.9\linewidth]{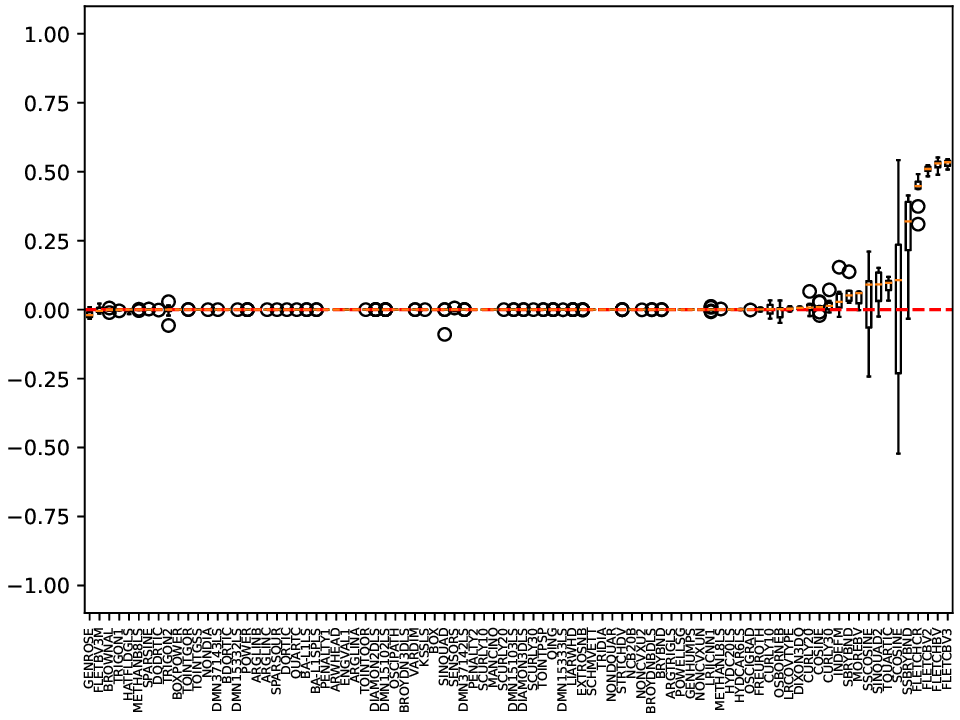}
     \caption{Values of $r_{i,p}$ (see \eqref{eq:ratio_def}) \replace{for}{shown as box-whisker plots on 30 random replications of applying \Cref{alg:practical} to} CUTEst problems of dimension $11\leq d\leq 100$. \textbf{Top: } Results with $p_k = \lceil 0.1 d\rceil$. \textbf{Bottom: } Results with $p_k = \lceil 0.01 d\rceil$.
\label{fig:alg2_100}}
\end{figure}

\begin{figure}
    \includegraphics[width=.9\linewidth]{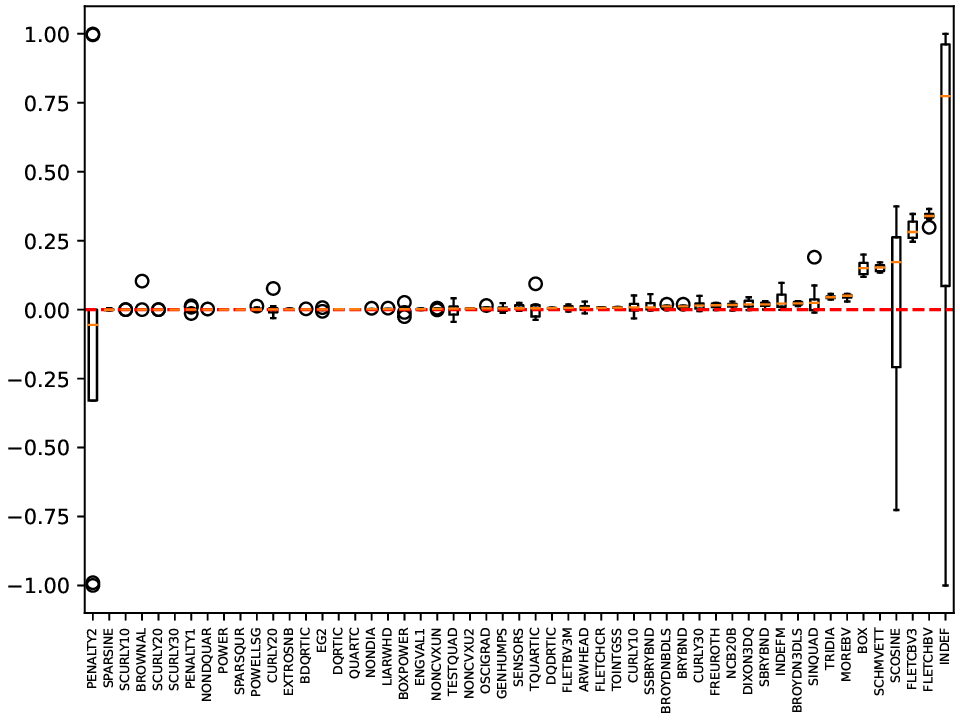}
      \includegraphics[width=.9\linewidth]{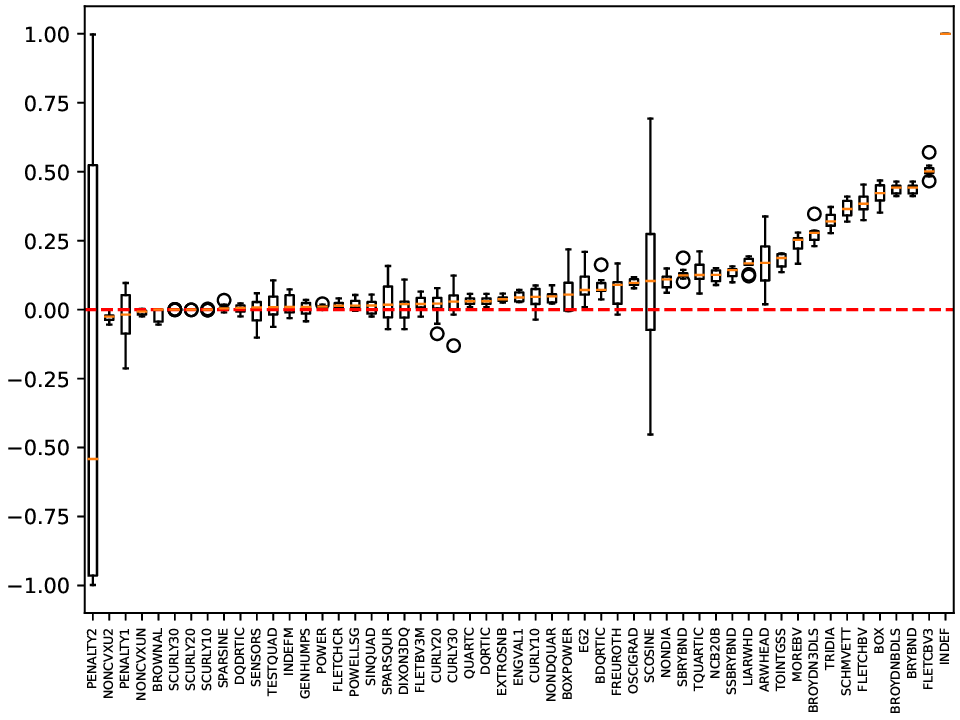}
      \caption{Values of $r_{i,p}$ (see \eqref{eq:ratio_def}) \replace{for}{shown as box-whisker plots on 30 random replications of applying \Cref{alg:practical} to} CUTEst problems of dimension $101\leq d\leq 1000$. \textbf{Top: } Results with $p_k = \lceil 0.01 d\rceil$. \textbf{Bottom: } Results with $p_k = \lceil 0.001 d\rceil$. 
\label{fig:alg2_1000}}
\end{figure}

\begin{figure}
        \includegraphics[width=.9\linewidth]{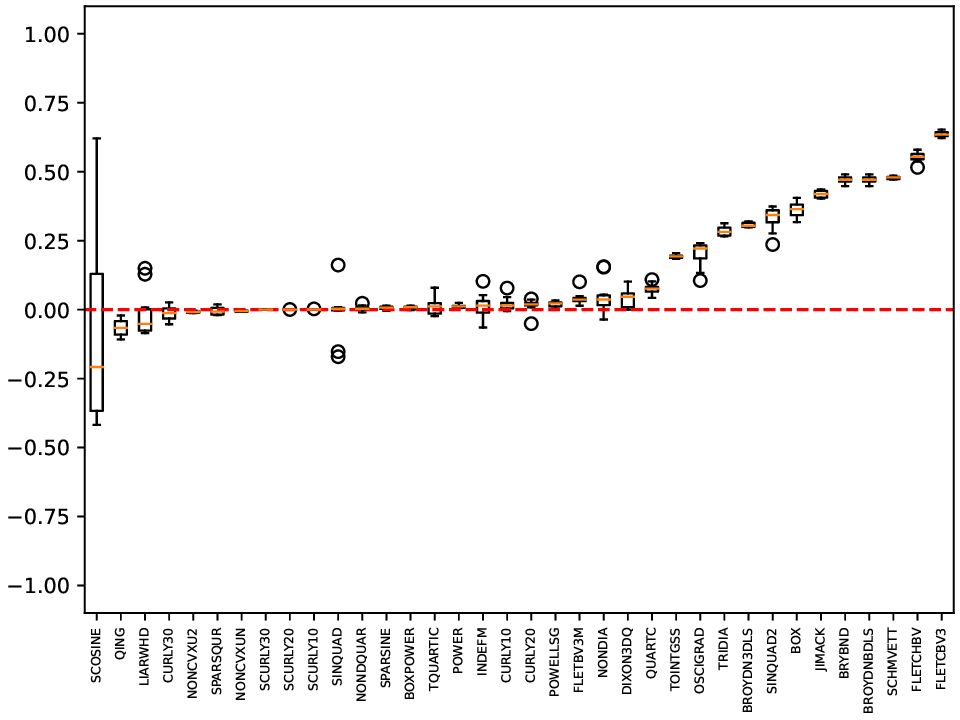}
    \includegraphics[width=.9\linewidth]{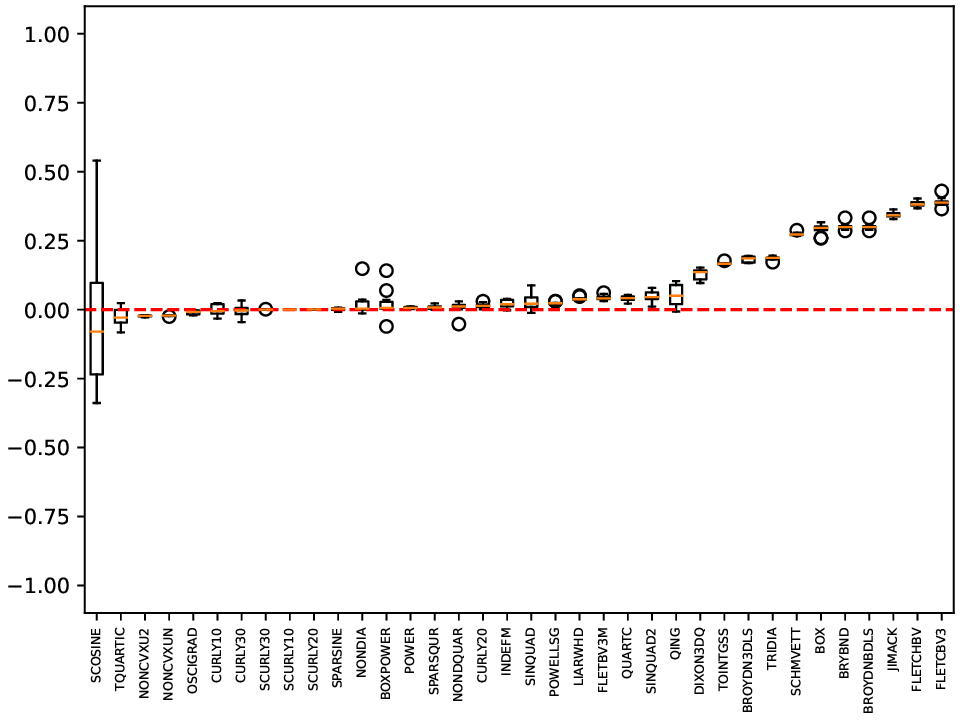}
    \caption{Values of $r_{i,p}$ (see \eqref{eq:ratio_def}) \replace{for}{shown as box-whisker plots on 30 random replications of applying \Cref{alg:practical} to} CUTEst problems of dimension $1001\leq d\leq 10000$. \textbf{Top: } Results with $p_k = \lceil 0.01 d\rceil$. \textbf{Bottom: } Results with $p_k = \lceil 0.001 d\rceil$. 
\label{fig:alg2_10000}}
\end{figure}

At a high level, 
\replace{\Cref{fig:alg2_100}, \Cref{fig:alg2_1000} and \Cref{fig:alg2_10000}}
{\Cref{fig:alg2_100,,fig:alg2_1000,,fig:alg2_10000}}
illustrate what we had hoped to show---across all problem sizes, there is a preference for using UCB over random-only, in the sense that most ratios across all runs are greater than 0.0, with particular exceptions seen in one problem in the $101\leq d\leq 1000$ setting (\texttt{PENALTY2}) and one problem in the $1001\leq d \leq 10000$ setting (\texttt{SSCOSINE}). 
In general, we observe a trend that as the dimension $d$ increases, the utility of UCB increases.
Moreover, for a fixed problem set of similar dimension, UCB is generally more effective when smaller sample sizes $p_k$ are employed.

\subsection{Testing \texttt{SS-POUNDers}}

We tested our implementation of \texttt{SS-POUNDers}, which borrows heavily from subroutines found in the Matlab implementation of \texttt{POUNDers} available in \texttt{IBCDFO} \cite{osti_2382683}. 
We first demonstrate the performance of \texttt{SS-POUNDers} on all 44 of the unconstrained \texttt{YATSOp} \cite{YATSOPCode} problems labeled midscale (i.e., with $98\leq d\leq 125$). 
We note that the \texttt{YATSOp} problems are a Matlab implementation of a particular subset of nonlinear least squares problems from \texttt{CUTEst} \cite{Gould2014}. 
Like the \texttt{POUNDers} software from which \texttt{SS-POUNDers} is derived, the ``default" use case of \texttt{SS-POUNDers} is for such nonlinear least squares problems.
This specification to nonlinear least squares essentially only affects how model gradients and \replace{model}{} Hessians are computed from residual vector evaluations in \Cref{line:build_model} of \Cref{alg:ss_pounders}; the geometry point selection determined by the UCB mechanism in \Cref{line:ucb_part} of \Cref{alg:ss_pounders} is agnostic to this additional nonlinear least squares structure. 

Our choice of testing the so-called midscale problems in \texttt{YATSOp} is partially to demonstrate something perhaps unsurprising that has been previously observed in the literature in sketching-based DFO; see \cite{cartis2023scalable}.
In particular, in the absence of any effective low-dimensionality in an objective function, one should generally expect a full-space optimization method (\texttt{POUNDers}) to outperform a subspace optimization method (\texttt{SS-POUNDers}), provided the performance metric is given by the number of function evaluations (as opposed to, say, wall-clock time) and further provided that a sufficiently large budget is made available.

We tested three variants of \texttt{SS-POUNDers} for illustrative purposes. 
The first variant, which we simply denote ``UCB" in the figure legends, performs \Cref{line:ucb_part} of \Cref{alg:ss_pounders} but effectively skips \Cref{line:choices} of \Cref{alg:ss_pounders} by always setting $p_k=0$. 
Note that this UCB variant, like \texttt{POUNDers}, is a \emph{deterministic} method. 
The second variant, which we denote ``random-only" in the figure legends, skips \Cref{line:ucb_part} of \Cref{alg:ss_pounders} entirely but sets $p_k=\min\{d-\dim(\span(\bS_k)),1\}$ in \Cref{line:choices} of \Cref{alg:ss_pounders}. 
This yields a fully randomized derivative-free subspace  optimization method like several of those discussed in our literature review but specialized to the \texttt{POUNDers} framework. 
Our third variant, denoted ``UCB + random" in the figure legends, performs  \Cref{line:ucb_part} of \Cref{alg:ss_pounders} and also sets $p_k=\min\{d-\dim(\span(\bS_k)),1\}$ in \Cref{line:choices} of \Cref{alg:ss_pounders}. 
This corresponds to the successful UCB variant of \Cref{alg:practical} from \Cref{sec:testing_sdm}. 

We show comparisons in terms of data profiles \cite{JJMSMW09}.
Given a set of problems $\mathcal{P}$ (each problem $prob$ is assumed to have a common initial point $\bx^{prob}_0$) and a set of solvers~$\mathcal{S}$, denote the sets of points evaluated by a solver $sol\in\mathcal{S}$ when solving an instance of problem $prob\in\mathcal{P}$, $\{\bx^{sol, prob}_k\}_{k=1}^{K}$. 
Given a tolerance $\tau\in(0,1)$, we say that a solver~$sol\in\mathcal{S}$ solved a problem $prob\in\mathcal{P}$ to tolerance $\tau$ within $N\leq K$ iterations provided

$$f(\bx^{sol, prob}_N) \leq \tau \left[f(\bx^{prob}_0) - \displaystyle\min_{sol^'\in\mathcal{S}} \displaystyle\min_{k=0,1,\dots,K} f(\bx^{sol^', prob}_k)\right].$$

A data profile then fixes a tolerance $\tau$ and generates a line for each solver $sol\in\mathcal{S}$.
On the $x$-axis, a data profile illustrates the number of budget units expended (in this manuscript, these budget units will always be groups of $d+2$ many function evaluations); and on the $y$-axis, a data profile shows the percentage of problems solved by solver $sol$ to tolerance $\tau$ within the corresponding budget. 
The choice of $d+2$ is because, in the absence of prior function evaluations being provided, \texttt{POUNDers} requires~$d+1$ function evaluations to construct an initial affine model and then requires one additional evaluation to evaluate the trial point that would be suggested by that initial model in the first iteration. 
That is, a function evaluation budget of $d+2$ guarantees  \texttt{POUNDers} can complete one full iteration in the absence of prior evaluations. 
For each problem/solver combination, we ran 10 fixed seeds of each of the randomized methods (the ``random-only" and ``UCB + random" variants), and we treated each seed as its own problem instance. 

\begin{figure}
    \centering
    \includegraphics[width=.99\textwidth]{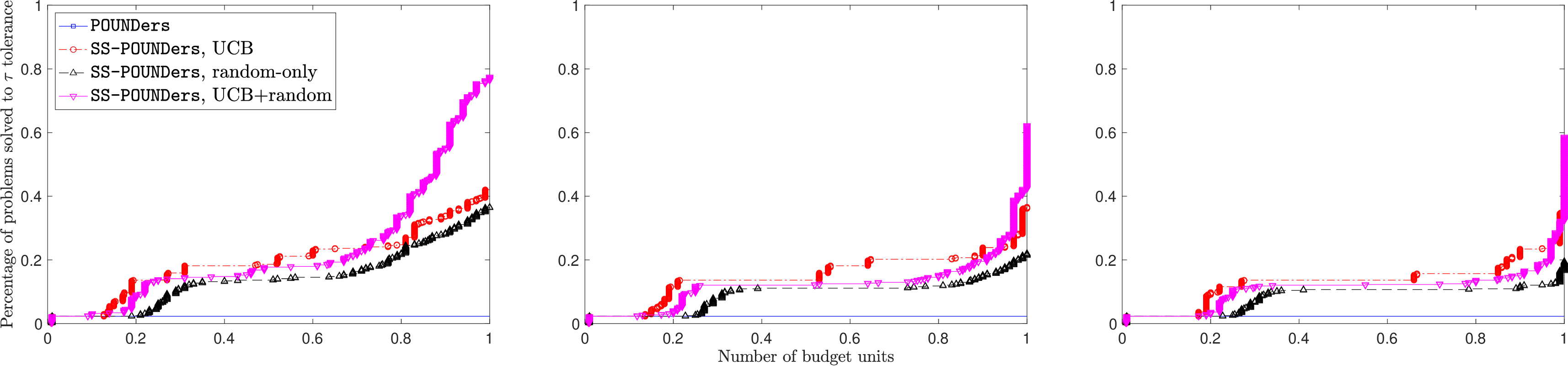}
    \caption{Data profiles comparing variants of \texttt{SS-POUNDers} with \texttt{POUNDers} in the budget-constrained setting of \emph{one} budget unit ($d+2$ function evaluations) on midscale \texttt{YATSOp} problems. Left figure is tolerance $\tau=0.1$, center figure is $\tau=0.01$, and right figure is $\tau=0.001$.\label{fig:yatsop_one_gradient}}
\end{figure}

\begin{figure}
    \centering
    \includegraphics[width=.99\textwidth]{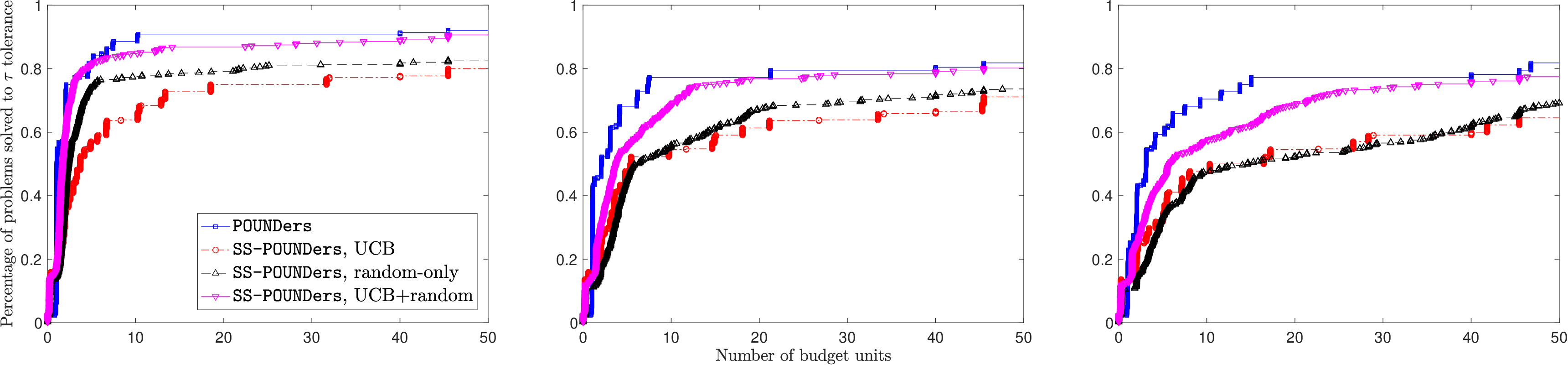}
      \caption{Data profiles comparing variants of \texttt{SS-POUNDers} with \texttt{POUNDers} in the less budget-constrained setting of \emph{fifty} budget units ($50(d+2)$ function evaluations) on midscale \texttt{YATSOp} problems. Left figure is tolerance $\tau=0.1$, center figure is $\tau=0.01$, and right figure is $\tau=0.001$.\label{fig:yatsop_fifty_gradients}}
\end{figure}

Once again, we did not perform any exhaustive hyperparameter tuning before performing these experiments. 
All default parameter settings common to both \texttt{POUNDers} and \texttt{SS-POUNDers} were set to the default parameter settings used in \texttt{POUNDers} in \texttt{IBCDFO}.
All of the parameters related to the UCB mechanism were set identically to those employed in \Cref{sec:testing_sdm}. 

As in the experiments of \Cref{sec:testing_sdm}, we first demonstrate results for a very budget-constrained computational setting, in particular, when only $d+2$ function evaluations---that is, one budget unit---is allotted. 
Unsurprisingly, in this extremely budget-constrained budget, and as illustrated in \Cref{fig:yatsop_one_gradient}, \texttt{POUNDers} performs  poorly compared with all of the \texttt{SS-POUNDers} variants. 
Once the budget is increased to a level where \texttt{POUNDers} is able to identify a reasonable solution ($50\replacemath{*}{}(d+2)$ function evaluations in these experiments), a preference for using \texttt{POUNDers} re-emerges, especially when we demand tighter tolerances of $\tau=0.01$ or $\tau=0.001$; see \Cref{fig:yatsop_fifty_gradients}. 
Remarkably, however, at these tolerances the variant of \texttt{SS-POUNDers} employing both the UCB mechanism and one-dimensional random sketches is not considerably worse than \texttt{POUNDers} in the latter setting. 
From this particular experiment, however, we do see a general preference for the UCB+random variant of \texttt{SS-POUNDers} over the other two variants of \Cref{alg:ss_pounders}. 
Importantly, we see no evidence that one should employ a fully randomized variant of \texttt{SS-POUNDers} over a variant that employs both a UCB mechanism and randomized sketching. 

Next we test a hypothesis made earlier in this manuscript that employing a UCB mechanism should exhibit stronger performance than a full-space method on functions exhibiting low effective dimensionality. 
For this test we  use  the low-dimensional~($2\leq~d\leq 12$) Mor{\'e}--Wild benchmarking set for DFO \cite{JJMSMW09}.
Following the experimental setup  in \cite{cartis2022randomised} and originally proposed in \cite{wang2016bayesian}, 
we specify a dimension $D$ (in the tests illustrated here, $D=100$) and artificially construct objective functions with $D$-dimensional domains starting from the benchmark problems. 
This is done by taking a function $f:\Reals^d\to\Reals$ from the benchmark set and defining a new function~$\bar{f}:\Reals^D\to\Reals$.
Specifically, we first define an intermediate function~$\tilde{f}:\Reals^D\to\Reals$ that acts trivially on the last $D-d$ dimensions of its domain, that is, 
$\tilde{f}(\bx) = f(\left[\bx_1,\dots, \bx_d\right]^\top)$.
We then define $\bar{f} = \tilde{f}(\bQ\bx)$, where $\bQ\in\Reals^{D\times D}$ is a randomly generated orthogonal matrix; this rotation is done to prevent the subspace of low effective dimensionality from being always coordinate-aligned. 
Comparisons of the same variants of \texttt{SS-POUNDers} and the same experimental setup as in the prior set of experiments are shown in \Cref{fig:low_effective_dimensionality}. 
We see that in this setting of low effective dimensionality, subspace optimization methods are clearly superior to the full-space \texttt{POUNDers}. 
Notably, at all three tolerances shown, the data profiles suggest a slight preference for using the deterministic UCB variant.
This observation is reasonable given that there is a constant subspace of variation that the UCB mechanism can learn, and so the subspace method is more likely to thrive when the balance of exploration vs exploitation is tipped to the direction of exploitation, since exploration in the orthogonal subspace defined by the last $D-d$ columns of the random matrix $\bQ$ will never align with $\nabla \bar f(\bx)$. 

\section{Conclusions and Future Work}
In this manuscript we introduced a linear UCB mechanism for dynamically learning subspaces well aligned with gradients in subspace optimization methods of the form presented in Framework~1. 
Subspace optimization methods like these are particularly useful in settings of computationally expensive optimization, where gradient information is prohibitive to obtain or approximate, relative to available computational budgets. 
The linear UCB mechanism is particularly appropriate for subspace optimization methods because it translates naturally to the problem of maximizing a linear function where the reward function (defined by the gradient) changes. 
We implemented two solvers, pseudocode provided in \Cref{alg:practical} and \Cref{alg:ss_pounders}, respectively, for problem settings where arbitrary directional derivatives are computable through a forward-mode AD oracle and for problems where no derivatives are available. 
Our experiments demonstrated the advantage of subspace optimization methods employing a linear UCB mechanism over subspace optimization methods that solely employ randomization.
Our experiments also demonstrated the advantage of subspace optimization methods over full-space methods in settings that are  extremely compute-limited (that is, only one or two full gradients or gradient approximations can be computed within budget) and in settings where objective functions exhibit low effective dimensionality. 

This work leaves open avenues for future research. 
In terms of theory, the regret bounds that we proved using existing techniques from the literature on bandit methods assume that the optimization method employed in Framework~2 results in a sequence of incumbents (and gradients) that are effectively exogenous from the perspective of the linear UCB mechanism. 
Identifying a means to incorporate the optimization dynamics directly into the regret analysis could potentially provide stronger guarantees that further elucidate the advantage of balancing exploration and exploitation via a UCB mechanism over methods that purely randomize subspace selection. 
As identified in the literature review, there is ongoing work to extend subspace optimization methods that fit in Framework~1 to problems exhibiting more structure than the unconstrained optimization problem presented in \eqref{eq:minf}. 
Extending the UCB mechanism to broader classes of problems, in particular constrained problems, could be of practical benefit in problems where subspaces must intersect nontrivially with feasible regions. 
Furthermore, we essentially conducted our regret analysis and proposed algorithms in this paper based  only on first-order optimality guarantees. 
In settings where Hessians can also be sketched or approximately sketched, one may try to employ a bandit mechanism that can incorporate (sketched) curvature information to obtain tighter confidence ellipsoids than those provided by a linear UCB mechanism that simply seeks to estimate the dynamically changing gradient.

\begin{figure}
    \centering
    \includegraphics[width=.99\textwidth]{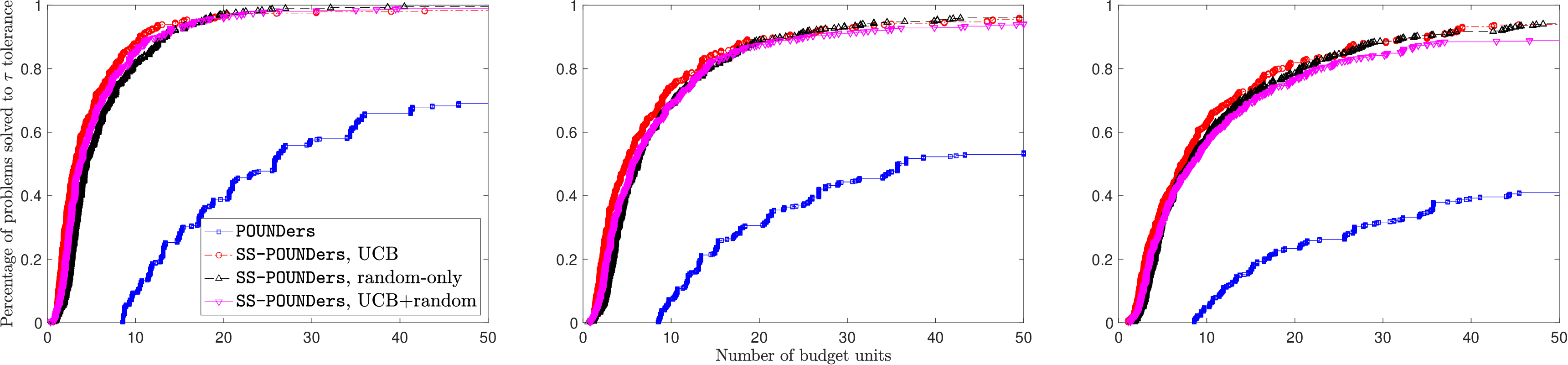}
       \caption{Data profiles comparing variants of \texttt{SS-POUNDers} with \texttt{POUNDers} with fifty budget units ($50(D+2)$ function evaluations) on the artificial problems of low effective dimensionality generated from the Mor{\'e}--Wild benchmarking set. Left figure is tolerance $\tau=0.1$, center figure is $\tau=0.01$, and right figure is $\tau=0.001$.\label{fig:low_effective_dimensionality}}
\end{figure}


\section*{Funding} 
This work was supported in part by the U.S.
Department of Energy, Office of Science, Office
of Advanced Scientific Computing Research, Scientific Discovery through Advanced Computing (SciDAC) Programs through Contract
Nos.\ DE-AC02-06CH11357 and through the FASTMath Institute. 

\section*{Disclosure statement}
The authors report there are no competing interests to declare.

\section*{Data availability statement}
The code used in these experiments will be made available upon reasonable request by contacting the author. 


\begin{appendix}
\section{Proof of \Cref{lem:gradient_error}}
\begin{proof}
    By the definition of the gradient estimate in \cref{eq:gradient_estimate} and the definition of $r_k$ in \cref{eq:sketch_gradient}, 
    $$
    \begin{array}{rl}
    g_k - \nabla f(x_k) 
    & = 
    C_k^{-1}\left[ \displaystyle\sum_{j=\max\{1, k-M-1\}}^{k-1} s_j s_j^\top \nabla f(x_j)\right] - \nabla f(x_k)\\
    & = 
     C_k^{-1}\left[ \displaystyle\sum_{j=\max\{1, k-M-1\}}^{k-1} s_j s_j^\top \nabla f(x_j) - C_k\nabla f(x_k)\right]\\
     &=
     C_k^{-1}\left[
     \left(
     \displaystyle\sum_{j=\max\{1, k-M-1\}}^{k-1} s_j s_j^\top (\nabla f(x_j) - \nabla f(x_k))
     \right)
     - \lambda \nabla f(x_k)
     \right].
    \end{array}
    $$
    By Cauchy--Schwarz inequality, for all $s\in\Reals^d$, 
    \begin{equation}\label{eq:intermediate1}
    |s^\top (g_k - \nabla f(x_k))| \leq  B_k\|s\| + \|s\|_{C_k^{-1}} \|\lambda C_k^{-1} \nabla f(x_k)\|_{C_k} =  B_k\|s\|  + \lambda\|s\|_{C_k^{-1}}\|\nabla f(x_k)\|_{C_k^{-1}},
    \end{equation}
    where we define the \emph{bias term} $B_k$ by
    $$B_k = \left\|C_k^{-1}\left(\displaystyle\sum_{j=\max\{1, k-M-1\}}^{k-1} s_j s_j^\top (\nabla f(x_j) - \nabla f(x_k))
     \right)
     \right\|.$$

    We note in \cref{eq:intermediate1} that 
    because $C_k \succeq \lambda I_d$, 
    we have that 
    $\|\nabla f(x_k)\|_{C_k^{-1}} \leq \lambda^{-1/2} \|\nabla f(x_k)\|$, and so we can simplify to
\begin{equation}\label{eq:intermediate2}
    |s^\top (g_k - \nabla f(x_k))| \leq  
    B_k\|s\|  + \sqrt{\lambda}\|\nabla f(x_k)\|\|s\|_{C_k^{-1}}.
    \end{equation}

    The remainder of the proof involves bounding the bias term.
    Abbreviate $\ell(k) = \max\{1,k-M-1\}$. 
    We first record that
    $$
    \begin{array}{rl}
    B_k =
    \left\|
    C_k^{-1}\left[\displaystyle\sum_{j=\ell(k)}^{k-1} s_j s_j^\top (\nabla f(x_j) - \nabla f(x_k))
     \right]
     \right\|
    & =
    \left\|
    C_k^{-1} \left[
    \displaystyle\sum_{j=\ell(k)}^{k-1} s_js_j^\top \left[
    \sum_{i=j}^{k-1}(\nabla f(x_i) - \nabla f(x_{i+1}))
    \right]
    \right]
    \right\|\\
    & = 
    \left\|
    C_k^{-1} \left[
    \displaystyle\sum_{i=\ell(k)}^{k-1}\left[
    \sum_{j=\ell(k)}^i s_j s_j^\top (\nabla f(x_i) - \nabla f(x_{i+1}))
    \right]    
    \right]
    \right\|\\
    & \leq
    \displaystyle\sum_{i=\ell(k)}^{k-1}\left\|
    C_k^{-1}\left[
    \sum_{j=\ell(k)}^i s_j s_j^\top (\nabla f(x_i) - \nabla f(x_{i+1}))
    \right]
    \right\|\\
    & \leq
     \displaystyle\sum_{i=\ell(k)}^{k-1}\left\|C_k^{-1}
        \sum_{j=\ell(k)}^i s_js_j^\top 
     \right\|
     \|\nabla f(x_i) - \nabla f(x_{i+1})\|.
    \end{array}
    $$  

Abbreviate $A_i := \displaystyle\sum_{j=\ell(k)}^i s_j s_j^\top.$
Observing that $C_k^{-1}$ and $A_i$ need not generally commute, we use the general matrix $2$-norm definition to derive that
$$
\begin{array}{rl}
\left\|C_k^{-1}A_i\right\| 
&=
\displaystyle\sup_{y:\|y\|=1}\sup_{z:\|z\|=1}
\left|
y^\top C_k^{-1}A_i z
\right|
 := y_*^\top C_k^{-1}A_i z_* \\
 & \leq 
 \|y_*\|_{C_k^{-1}} \|A_i z_*\|_{C_k^{-1}}\\
 & \leq 
 \|y_*\|_{C_k^{-1}} \left\|\displaystyle\sum_{j=\ell(k)}^i s_j \|s_j\|\|z_*\| \right\|_{C_k^{-1}}.
\end{array}
$$
Because $\|s_j\|=1$ for all $j$, and because $\|y_*\|_{C_k^{-1}} \leq \lambda^{-1/2}\|y_*\| = \lambda^{-1/2}$, this simplifies to
$$
\begin{array}{rl}
\left\|C_k^{-1}A_i\right\| 
 \leq 
\lambda^{-1/2}\left\|
\displaystyle\sum_{j=\ell(k)}^i s_j
\right\|_{C_k^{-1}}
\leq 
\lambda^{-1/2}\displaystyle\sum_{j=\ell(k)}^i \left\|s_j\right\|_{C_k^{-1}}
& \leq 
\lambda^{-1/2} \displaystyle\sqrt{i - \ell(k) + 1} 
\displaystyle\sqrt{\sum_{j=\ell(k)}^i \| s_j\|^2_{C_k^{-1}}},
\end{array}
$$
where the last step employs Cauchy--Schwarz inequality. 
Finally,
$$
\begin{array}{rl}
\displaystyle\sum_{j=\ell(k)}^i \| s_j\|^2_{C_k^{-1}}
& = 
Tr\left(C_k^{-1}\displaystyle\sum_{j=\ell(k)}^i s_js_j^\top \right)\\
& \leq 
Tr\left(C_k^{-1}\displaystyle\sum_{j=\ell(k)}^i s_js_j^\top\right)
+ \displaystyle\sum_{j=i+1}^{k-1} s_j^\top C_k^{-1} s_j 
+ \lambda \displaystyle\sum_{\ell=1}^d e_i^\top C_k^{-1} e_i \\
& = 
Tr\left(C_k^{-1}\displaystyle\sum_{j=\ell(k)}^i s_js_j^\top\right)
+ Tr\left(C_k^{-1}\displaystyle\sum_{j=1+1}^{k-1} s_js_j^\top\right)
+ Tr\left(C_k^{-1} \lambda \displaystyle\sum_{\ell=1}^d e_i e_i^\top\right)\\
& = Tr(C_k^{-1}(A_{k-1} + \lambda I_d)) = Tr(C_k^{-1} C_k) = Tr(I_d) = d,\\
\end{array}
$$
and so
$$
\begin{array}{rl}
B_k \leq 
\displaystyle\sum_{i=\ell(k)}^{k-1} 
 \displaystyle\sqrt{\frac{d(i - \ell(k) + 1)}{\lambda}}
  \|\nabla f(x_{i+1}) - \nabla f(x_i)\|
& \leq
  \displaystyle\sqrt{\frac{d(M+1)}{\lambda}}
 \displaystyle\sum_{i=\ell(k)}^{k-1} \|\nabla f(x_{i+1}) - \nabla f(x_i)\|.\\
 \end{array}
$$
The claim follows. 
\end{proof}

\section{Proof of \Cref{lem:potential_lemma}}
\begin{proof}
We begin by constructing a set of auxiliary matrices. 
Without loss of generality, suppose that $k/M = B$ for some integer $B$. 
We will consider a partition of $\{1,2,\dots,k\}$ into $B$ consecutive blocks.
For the $b$th of the $B$ blocks and for each $j\in\{1,2,\dots,k\}$, we define the \emph{auxiliary matrix}

$$A_j^b = \lambda I_d + \displaystyle\sum_{i=(b-1)M + 1}^{j-1} s_is_i^\top.$$

We make the key observation that for all $j\in\{(b-1)M + 1,\dots, bM\}$, we have the relation $A_j^b \preceq C_j$; 
this follows since $C_j$ is defined as 
$$C_j = \lambda I_d + \displaystyle\sum_{i=\max\{1, j-M-1\}}^{j-1} s_i s_i^\top;$$  
thus, $C_j$ can be viewed as the sum of $A_j^b$ with, potentially, some additional rank-one matrices. 
Moreover, because both $A_j^b$ and $C_j$ are positive definite matrices, we have the relation
$C_j^{-1} \preceq (A_j^b)^{-1}$, 
and so
$$
\displaystyle\sum_{j=1}^k \|C_{j-1}^{-\frac{1}{2}}s_j\| 
\leq 
\displaystyle\sum_{j=1}^k \|(A_{j-1}^b)^{-\frac{1}{2}}s_j\|.
$$

For all $j\in\{(b-1)M+1,\dots,bM\}$, we can rewrite 
$$A_j^b = A_{j-1}^b + s_js_j^\top  = (A_{j-1}^b)^{\frac{1}{2}}\left(
I_d 
+ (A_{j-1}^b)^{-\frac{1}{2}}s_js_j^\top (A_{j-1}^b)^{-\frac{1}{2}}
\right)
(A_{j-1}^b)^{\frac{1}{2}}.$$
Taking the determinant on both sides and using the permutation property of determinants, we have
$$\det(A_j^b) = \det(A_{j-1}^b)\det\left(
I_d + (A_{j-1}^b)^{-\frac{1}{2}}s_js_j^\top (A_{j-1}^b)^{-\frac{1}{2}}
\right).
$$ 
By a standard result of rank-one updates and using the inequality $1 + y \geq \exp(y/2)$ for $y\in[0,1]$,  we have
$$\det(A_j^b) = \det(A_{j-1}^b)(1 + \|(A_{j-1}^b)^{-\frac{1}{2}}s_j\|^2) 
\geq
\det(A_{j-1}^b)\exp(\|(A_{j-1}^b)^{-\frac{1}{2}}s_j\|^2/2).
$$
Rearranging,
$$
\|(A_{j-1}^b)^{-\frac{1}{2}}s_j\|^2
\leq
2\log\left(\displaystyle\frac{\det(A_j^b)}{\det(A_{j-1}^b)}\right),$$
and so 
$$\displaystyle\sum_{j=(b-1)M+1}^{bM} 
\|(A_{j-1}^b)^{-\frac{1}{2}}s_j\|^2
\leq
2 \displaystyle\sum_{j=(b-1)M+1}^{bM} \log\left(\displaystyle\frac{\det(A_j^b)}{\det(A_{j-1}^b)}\right)
= 2\log\left(\displaystyle\frac{\det(A_{bM}^b)}{\det(A_{(b-1)M+1}^b)}\right).
$$
Notice that by the definition of $A_{bM}^b$ and because each $s_j$ is a unit vector, 
$$\Tr(A_{bM}^b) = \Tr(A_{(b-1)M+1}^b) + \displaystyle\sum_{j=(b-1)M+1}^{bM} \Tr(s_js_j^\top) =\lambda d + M,$$
and so
$\det(A_{bM}^b)\leq (\lambda + (M/d))^d$. 
Thus,
$$
\begin{array}{rl}
\displaystyle\sum_{j=1}^k 
\|C_{j-1}^{-\frac{1}{2}}s_j\|^2
=
\displaystyle\sum_{b=1}^{k/M} 
\displaystyle\sum_{j=(b-1)M+1}^{bM}
\|C_{j-1}^{-\frac{1}{2}}s_j\|^2
& \leq 
\displaystyle\sum_{b=1}^{k/M} 
\displaystyle\sum_{j=(b-1)M+1}^{bM}
\|(A_{j-1}^b)^{-\frac{1}{2}}s_j\|^2\\
& \leq
\displaystyle\sum_{b=1}^{k/M}
2\log\left(\left(\displaystyle\frac{\lambda +M/d}{\lambda} \right)^d\right) \\
& = 
\displaystyle\frac{2kd}{M}\log\left(1 + \displaystyle\frac{M}{\lambda d}\right).
\end{array}
$$

\end{proof}
    \end{appendix}

\framebox{\parbox{\columnwidth}{The submitted manuscript has been created by UChicago Argonne, LLC, Operator of Argonne National Laboratory (`Argonne'). Argonne, a U.S. Department of Energy Office of Science laboratory, is operated under Contract No. DE-AC02-06CH11357. The U.S. Government retains for itself, and others acting on its behalf, a paid-up nonexclusive, irrevocable worldwide license in said article to reproduce, prepare derivative works, distribute copies to the public, and perform publicly and display publicly, by or on behalf of the Government.  The Department of Energy will provide public access to these results of federally sponsored research in accordance with the DOE Public Access Plan. \url{http://energy.gov/downloads/doe-public-access-plan}.}}
\end{document}